\documentclass[11pt]{amsart}
\usepackage{amssymb}
\usepackage{mathtools}
\usepackage[margin=3cm,marginparwidth=15pt]{geometry}
\usepackage[UKenglish]{babel} 
\usepackage[colorlinks=true,linkcolor=blue,citecolor=magenta]{hyperref}
\usepackage[hyphenbreaks]{breakurl} 
\usepackage[shortlabels]{enumitem} 
\usepackage{csquotes}
\usepackage[capitalise]{cleveref}
\usepackage[normalem]{ulem}

\usepackage{tikz}

\newcommand\F{\mathbb{F}}
\newcommand\Fq{\mathbb{F}_q}

\newcommand\Q{\mathbb{Q}}
\newcommand\Z{\mathbb{Z}}
\def\Jac{\mathop{\rm Jac}\nolimits}

\def\rk{\mathop{\rm rk}\nolimits}
\def\Hom{\mathop{\rm Hom}\nolimits}
\def\End{\mathop{\rm End}\nolimits}

\def\Gal{\mathop{\rm Gal}\nolimits}
\def\Jac{\mathop{\rm Jac}\nolimits}
\def\Mod{\mathop{\rm Mod}\nolimits}
\def\Pr{\mathop{\rm Pr}\nolimits}
\def\KerC{\mathop{\rm Ker_\mathcal{C}}\nolimits}

\newcommand\cB{\mathcal{B}}
\newcommand\cC{\mathcal{C}}

\newcommand\cO{\mathcal{O}}

\newcommand\frl{\mathfrak{l}}
\newcommand\frm{\mathfrak{m}}
\newcommand\frn{\mathfrak{n}}
\newcommand\frL{\mathfrak{L}}
\newcommand\frM{\mathfrak{M}}
\newcommand\frN{\mathfrak{N}}

\newcommand\into{\hookrightarrow}
\newcommand\onto{\twoheadrightarrow}

\theoremstyle{plain}
\newtheorem{maintheorem}{Main Theorem}
\newtheorem*{cor_intro}{Corollary}
\newtheorem{theorem}{Theorem}[section]
\newtheorem{lemma}[theorem]{Lemma}
\newtheorem{proposition}[theorem]{Proposition}
\newtheorem{corollary}[theorem]{Corollary}
\theoremstyle{definition}
\newtheorem{definition}[theorem]{Definition}
\newtheorem{remark}[theorem]{Remark}

\newtheorem{example}[theorem]{Example}

\newcommand\PirrA{\mathcal{P}^\mathrm{irr}_\mathrm{npp}}
\newcommand\PirrB{\mathcal{P}^\mathrm{irr}_\mathrm{Wres}}

\title{Abelian surfaces over finite fields containing no curves of genus $3$ or less}
\author{Elena Berardini}
\address{CNRS; IMB, Université de Bordeaux, 351 cours de la Libération, 33405 Talence, France}
\email{elena.berardini@math.u-bordeaux.fr}
\author{Alejandro Giangreco Maidana}
\address{Facultad de Ingeniería, Universidad Nacional de Asunción}
\email{agiangreco@ing.una.py}
\author{Stefano Marseglia}
\address{Mathematisch Instituut, Universiteit Utrecht, Postbus 80010, 3508 TA Utrecht, The Netherlands}
\address{Laboratoire GAATI, Université de la Polynésie française, BP 6570 -- 98702 Faaa, Polynésie française}
\address{Laboratoire Jean Alexandre Dieudonn\`e, Universit\`e C\^ote d'Azur, 06108 Nice Cedex 2, France}
\email{stefano.marseglia@univ-cotedazur.fr}


\begin{document}

\begin{abstract}
    We study abelian surfaces defined over finite fields which do not contain any possibly singular curve of genus less than or equal to $3$. 
    Firstly, we complete and expand the characterisation of isogeny classes of abelian surfaces with no curves of genus up to $2$ initiated by the first author \emph{et al.~}in previous work.
    Secondly, we show that, for simple abelian surfaces, containing a curve of genus $3$ is equivalent to admitting a polarisation of degree $4$. Thanks to this result, we can use existing algorithms to check which isomorphism classes in the isogeny classes containing no genus $2$ curves have a polarisation of degree $4$. 
    Thirdly, we characterise isogeny classes of abelian surfaces with no curves of genus $\leq 2$, containing no abelian surface with a polarisation of degree $4$. 
    Finally, we describe the absolutely irreducible genus $3$ curves lying on abelian surfaces containing  no curves of genus less than or equal to~$2$.
\end{abstract}

\maketitle
\section*{Introduction}
Studying the minimal genus of algebraic curves lying on an abelian variety is a classical question \cite{AP90,debarre1994degrees,BCV95}. 
In this paper, we focus our attention on abelian surfaces.
Over an algebraically closed field, every abelian surface is isogenous to a principally polarised one  \cite[Cor.~1, p.~234]{M_AV70}, that is, to the Jacobian of a smooth genus 2 curve or to the product of two elliptic curves. 
A well-known consequence of this fact is that every abelian surface over an algebraically closed field contains a (possibly singular) irreducible curve of geometric genus $2$ or $1$; see also Lemma \ref{lem:ggg} below.  On another note, one can also study when an abelian surface contains curves of a fixed genus. Over the complex numbers, this question has been treated in \cite{Barth} for genus $3$ curves and in \cite{BS17} for hyperelliptic curves of genus $4$. Over non-algebraically closed fields, the situation is more complicated. For example, it is no longer true that any abelian surface is isogenous to a principally polarised one. 

From now on, we direct our attention to the case of finite fields. Isogeny classes in which no abelian surface admits a principal polarisation are classified in \cite{HMNR08}. Furthermore, even if an isogeny class is principally polarisable, it might not contain a Jacobian surface nor the product of two elliptic curves \cite{HNR09}. Indeed, by a theorem of Weil \cite[Thm.~1.3]{HNR09}, we have one more option among principally polarisable abelian surfaces, that is, Weil restrictions of elliptic curves defined over a quadratic extension of the finite field.  To sum up,  studying the minimal genus of algebraic curves lying over an abelian surface defined over a finite field is a hard question, and it is in fact open in its full generality.

Besides the intrinsic theoretical interest, this question was raised in the context of coding theory. The first author and her co-authors studied in \cite{MYFE1} the so-called algebraic geometry codes constructed from a rational divisor on an abelian surface defined over a finite field $\Fq$. They showed that the minimum distance of such codes respects, under some hypotheses, a lower bound that is somewhat proportional to the minimum genus of the $\Fq$-irreducible curves defined over $\Fq$ lying on the surface \cite[Thm.~3.3 and Rmk.~3.6]{MYFE1}. Since codes with large minimum distance are desirable, this led to the question on how to find abelian surfaces containing no curves of low genus.

\subsection*{Our contributions.}
In the following, by \emph{curve} we mean an algebraic projective one-dimensional variety defined over a field. 
Note that for us a curve doesn't need to be smooth or irreducible.
In the rest of the paper, unless otherwise specified, we will assume the base field to be $\Fq$, a finite field with $q$ elements of characteristic $p$, and the curve to be $\Fq$-irreducible.
The definition of arithmetic and geometric genus of such a curve is recalled at the beginning of Section~\ref{sec:genus2}.

An isogeny class of abelian surfaces over $\Fq$ corresponds via the Honda--Tate theory to a polynomial of the form $f(t)=t^4+at^3+bt^2+aqt+q^2$ in $\Z[x]$, called Weil polynomial. 
Previous work by the first author and coauthors showed that abelian surfaces defined over $\Fq$ which do not contain absolutely irreducible curves of arithmetic genus up to $2$ are either not isogenous to a principally polarised abelian surface or are isogenous to Weil restrictions of some elliptic curves defined over $\mathbb{F}_{q^2}$; see \cite[Prop.~4.2 and 4.3]{MYFE1}.  
We denote the set of Weil polynomials corresponding to the former isogeny classes by $\PirrA$ while the latter classes splits into $\{ (t^2-2)^2 , (t^2-3)^2 \} \sqcup\PirrB$; see Definition \ref{df:partition} for the precise definition of these classes. 
As the notation suggests, the Weil polynomials of the aforementioned isogeny classes are all irreducible except for $(t^2-2)^2$ and $(t^2-3)^2$.
These results are expanded in Main Theorem~\ref{mainthm0}.

\begin{maintheorem}\label{mainthm0}
    Let $f(t)$ be a Weil polynomial corresponding to an isogeny class $\cC$ of simple abelian surfaces over $\Fq$.
    The following statement holds:
    \begin{enumerate}[(i)]
        \item \label{mainthm0:abs_irr_g_leq2} No $A$ in $\cC$ contains an absolutely irreducible curve of geometric genus $\leq 2$ if and only if $f(t)$ is in $\PirrA \sqcup \{ (t^2-2)^2 , (t^2-3)^2 \} \sqcup\PirrB$.
        \item \label{mainthm0:Pnpp} No $A$ in $\cC$ contains an $\Fq$-irreducible curve of arithmetic genus $\leq 2$ if and only if  $f(t)$ is in $\PirrA$.
        \item \label{mainthm0:PWres} No $A$ in $\cC$ contains an absolutely irreducible curve with geometric genus $\leq 2$ 
        but there exists a $B$ in $\cC$ which contains an $\Fq$-irreducible curve with arithmetic genus $2$  if and only if $f(t)$ is in $\{ (t^2-2)^2 , (t^2-3)^2 \} \sqcup\PirrB$.
    \end{enumerate}
\end{maintheorem}
An expanded version of Part~\ref{mainthm0:abs_irr_g_leq2} is in Theorem~\ref{thm:no_g1g2}, where we also give a characterisation of the Weil polynomials in terms of their coefficients.
In Remark \ref{rmk:gap} we show that restricting the statement of Part~\ref{mainthm0:abs_irr_g_leq2}  to absolutely irreducible curves of arithmetic genus up to $2$ does not provide an equivalence. Parts~\ref{mainthm0:Pnpp} and \ref{mainthm0:PWres} are proven as Corollary~\ref{cor:mainthm0}.

The rest of the paper aims at extending Main Theorem~\ref{mainthm0} to curves of genus $\leq 3$.
The main tool for this purpose will be our second main result, which can be found later in the text as Theorem~\ref{thm:pol_deg_4_iff_pa3}.
\begin{maintheorem}\label{mainthm1}
    Let $A$ be a simple abelian surface defined over $\Fq$. Then, the following are equivalent:
    \begin{enumerate}[(i)]
        \item $A$ has a polarisation of degree $4$,
        \item $A$ contains an $\Fq$-irreducible curve of arithmetic genus $3$.
    \end{enumerate}
\end{maintheorem}

Firstly, we treat the isogeny classes $\cC$ and $\cC'$ defined by the reducible Weil polynomials $(t^2-2)^2$ and $(t^2-3)^2$, respectively.
As we know by Main Theorem~\ref{mainthm0}.\ref{mainthm0:PWres}, both $\cC$ and $\cC'$ contain at least one abelian surface with an $\Fq$-irreducible -- but not absolutely irreducible -- curve of arithmetic genus $2$ lying on it.
So our focus is on absolutely irreducible curves of arithmetic genus $3$.
In Proposition~\ref{prop:q2q3}, we show that no abelian surface in $\cC$ contains an absolutely irreducible curve of geometric or arithmetic genus $3$, while there exists an abelian surface in $\cC'$ containing an absolutely irreducible smooth curve of genus $3$.

Secondly, we consider the irreducible Weil polynomials $f(t) \in \PirrA\sqcup\PirrB$.
When the corresponding isogeny class $\cC$ is ordinary, the third author designed an algorithm in \cite{Mar21} to compute the isomorphism classes of abelian surface in $\cC$ admitting a polarisation of degree $4$, and hence containing an $\Fq$-irreducible curve of arithmetic genus $3$.
See Section~\ref{s:examples} for an overview and examples.
By applying Main Theorem~\ref{mainthm1} to our isogeny classes, we easily obtain the following statements about curves of arithmetic genus $\leq 3$ (and not just $=3$).
See Corollary~\ref{cor:consequence} for the proof.
\begin{cor_intro}\label{cor_intro:consequences}
    Let $f(t)$ be the Weil polynomial of a simple isogeny class $\cC$ of abelian surfaces over $\Fq$.
    \begin{enumerate}[(i)]
        \item \label{mainthm1:consequence:1} 
            Assume that $f(t)\in\PirrA\sqcup\PirrB$ and let $A$ be an abelian surface in $\cC$.
            If $A$ contains an absolutely irreducible curve of arithmetic genus $\leq 3$ then $A$ admits a polarisation of degree $4$.
        \item \label{mainthm1:consequence:2}
            No abelian surface in $\cC$ contains an $\Fq$-irreducible curve of arithmetic genus $\leq 3$ if and only if $f(t)$ is in $\PirrA$ and no abelian surface in $\cC$ admits a polarisation of degree~$4$.
    \end{enumerate}      
\end{cor_intro}

Finally, in Main Theorem~\ref{mainthm2} below, we give a complete classification of the isogeny classes determined by $f(t)\in \PirrA\sqcup\PirrB$ that contain an abelian surface with a polarisation of degree $4$.
The proof of this classification builds on Howe's seminal work \cite{Howe95,Howe96} on kernels of polarisations of abelian varieties over finite fields.
An expanded version of Main Theorem~\ref{mainthm2} appears later in the paper as Theorem~\ref{thm:kers_size_4}.
\begin{maintheorem}\label{mainthm2}
    Assume that $f(t)\in\PirrA\sqcup\PirrB$ and let $\cC$ be the corresponding isogeny class. 
    Set $K=\Q[t]/(f(t))$, which is a CM field with totally real subfield $K^+$. 
    Consider the statement:
    \begin{equation}\label{mainthm:pol}
    \text{There is no } A \text{ in } \cC \text{ admitting a polarisation of degree 4.} \tag{$\bigstar$}
    \end{equation}
    Then, the following statements hold.
    \begin{enumerate}[(i)]
        \item \label{mainthm:ord}
            Assume $\cC$ is ordinary. Then \eqref{mainthm:pol} holds if and only if there is no $A$ in $\cC$ with maximal $\Fq$-endomorphism ring admitting a polarisation of degree $4$.
        \item \label{mainthm:kers_size_4:A}
            Assume that $f(t)\in\PirrA$.
            Then \eqref{mainthm:pol} holds if and only if $2$ is inert in $K^+$.
        \item \label{mainthm:kers_size_4:B}
            Assume that $f(t)\in \PirrB$, so that $f(t)=t^4+bt^2+q^2$. 
            Then \eqref{mainthm:pol} is equivalent to 
            \begin{itemize}
            \item $b=1-2q$ and $q$ is odd, if $\cC$ is ordinary,
            \item $q$ is even, if $\cC$ is non-ordinary.
            \end{itemize}
    \end{enumerate}
\end{maintheorem}

\subsection*{Organisation of the paper.}
In Section \ref{sec:genus2}, we prove Main Theorem~\ref{mainthm0}, expanding on \cite{MYFE1} and characterising the isogeny classes of abelian surfaces containing no absolutely irreducible curves of geometric genus smaller than or equal to $2$.
In Section \ref{sec:g3onsurf}, we start the characterisation of abelian surfaces containing no curves of genus $3$ among those containing no curves of genus up to $2$. 
In particular, in Theorem~\ref{thm:pol_deg_4_iff_pa3}, we prove the key equivalence between containing a curve of arithmetic genus $3$ and having a polarisation of degree $4$ stated above as Main Theorem~\ref{mainthm1}.
We derive also some consequences stated above in the Corollary.
In Section \ref{sec:fact_of_2} we collect technical results on the factorisation of $2$ in the extension $K/\Q$, that we shall use throughout the paper. 
Section \ref{sec:kernels} is devoted to results on kernels of polarisations of abelian varieties, based on the work of Howe \cite{Howe95,Howe96}. 
In Section \ref{sec:pols4}, we prove an expanded version of Main Theorem~\ref{mainthm2}, namely Theorem~\ref{thm:kers_size_4}, in which we give necessary and sufficient conditions on an isogeny class to not contain surfaces admitting a polarisation of degree $4$. 
Section \ref{s:examples} describes the algorithm proposed by the third author in \cite{Mar21} to compute the isomorphism classes of abelian surface in an isogeny class admitting a polarisation of degree $4$, and provides examples for our Theorem \ref{thm:kers_size_4}. 
Finally, in Section \ref{ses:curveside}, we describe absolutely irreducible smooth curves of genus $3$ lying on abelian surfaces containing no curves of genus less than $2$. In particular, we give bounds for their number of rational points, showing that such genus $3$ curves are far from being maximal.

\section{Abelian surfaces containing no curves of genus \texorpdfstring{$\leq 2$}{<= 2}}\label{sec:genus2}
Let $A$ be a $g$-dimensional abelian variety defined over a finite field $\Fq$ of characteristic $p$.
The characteristic polynomial $f(t)$ of its Frobenius endomorphism acting on the $\ell$-Tate module (for any prime $\ell\neq p$) is a monic polynomial of degree $2g$ with integer coefficients.
All the complex roots of $f(t)$ have absolute value $\sqrt{q}$.
We call such a polynomial a Weil polynomial.
Honda and Tate showed in \cite{tate1966,honda1968,Tate1971} that $f(t)$ completely determines the isogeny class $\cC$ of $A$.
Recall that an abelian variety $A$ over $\Fq$ of dimension $g$ is called ordinary if the coefficient of $t^g$ in $f(t)$ is coprime with $q$.
In particular, being ordinary is a property of the isogeny class $\cC$ of $A$.
We will write that $\cC$ and $f(t)$ are ordinary if $A$ is so. 

In the rest of the paper, whenever we talk about a morphism, we always mean an $\Fq$-morphism. In particular, we will say simple for $\Fq$-simple, isogeny for $\Fq$-isogeny, etc.

By a \emph{curve} we mean a possibly singular one-dimensional projective variety defined over $\Fq$. 
We will only consider $\Fq$-irreducible curves in this paper.
In some cases, we will need our curves to be absolutely irreducible, that is, irreducible over the algebraic closure $\overline\F_q$ of $\Fq$. 
We say that a curve $C$ defined over $\Fq$ lies on an abelian surface $A$ or that $A$ contains $C$, if $C$ is a closed subvariety of $A$, that is, there is an $\Fq$-embedding of $C$ in $A$. A \emph{divisor} of $A$ defined over $\Fq$, or a \emph{rational divisor} over $\Fq$, is a formal sum, with integer coefficients, of $\Fq$-irreducible curves defined over $\F_q$ and lying on $A$\footnote{A divisor is often called rational over $\Fq$ when it is invariant under the action of $\mathrm{Gal}(\overline\F_q/\Fq)$. This implies in particular that the divisor can be written as a formal sum of $\Fq$-irreducible curves defined over $\F_q$ (see \cite[\S 3.1]{Haloui17}). In the present paper we shall only use this characterization, hence we take it as a definition.}. It is called \emph{effective} if all its coefficients are non-negative and \emph{ample} if a multiple of it is \emph{very ample}, that is, defines an embedding of the surface in a projective space.

The arithmetic genus $p_a(C)$ of an $\Fq$-irreducible curve $C$ is defined as $1-\chi(C)$, $\chi(C)$ being the Euler-Poincaré characteristic of $C$  \cite[Ch.~IV, Sec.~2]{Serre12}.
Consider the absolutely irreducible components $C_i$ of $C$, and their normalisations $\tilde{C_i}$. We define the geometric genus $g(C)$ of $C$ as $g(C) = p_a(\tilde{C_1})+\ldots+p_a(\tilde{C_n})$. 
We have that $g(C) \leq p_a(C)$ with an equality if and only if $C$ is smooth.
If $C$ is an $\F_q$-irreducible curve lying on an abelian surface $A$ then, since the class of the canonical divisor of $A$ is trivial, we have that $C^2=2p_a(C)-2$ by the adjunction formula \cite[Ch.~IV, Sec.~2, Prop.~5]{Serre12}.
\smallskip
 
The goal of this section is to prove some characterisations and properties of Weil polynomials of abelian surfaces which do not contain absolutely irreducible curves of geometric genus $0$, $1$ or $2$. We start with the following useful lemma.

\begin{lemma}\label{lem:ggg}
    Let $f\colon A\to B$ be an isogeny of abelian varieties defined over $\Fq$. 
    Let $C$ be an $\Fq$-irreducible curve of geometric genus $g$ defined over $\Fq$, lying on $A$. 
    Then, $f(C)$ is an $\Fq$-irreducible curve of geometric genus $\leq g$ on $B$.
    Moreover, if $C$ is absolutely irreducible then $f(C)$ is absolutely irreducible as well.
\end{lemma}
\begin{proof}
    Set $D=f(C)$ and denote by $f_C \colon C\to D$ the restriction of $f$ to $C$.
    Since $f$ is an isogeny, then $D$ is an $\Fq$-irreducible curve.
    By the universal property of the normalisation $\tilde D\to D$, the composition of the normalisation $\tilde C\to C$ with $f_C\colon C\to D$ factors as the composition of a morphism $\tilde f_C\colon\tilde C\to \tilde D$ with $\tilde D\to D$. 
    The morphism $\tilde f_C$ factors as $\tilde C \overset{i}{\to} Y \overset{s}{\to} \tilde D$ with $Y$ a normal scheme, $i$ a purely inseparable morphism and $s$ a separable morphism. 
    Then, $g(\tilde C)=g(Y)$.
    Finally, 
    the Riemann-Hirwitz Theorem applied to $s:Y\to \tilde D$ gives us $g(\tilde D)\leq g(Y)$.
    Combining, we obtain $g(C)=g(\tilde C)=g(Y) \geq g(\tilde D) = g(D)$.
    
    For the final statement, assume that $C$ is absolutely irreducible and let $\bar C$ and $\bar f$ be the extension of $C$ and $f$, respectively, to an algebraic closure $\bar\Fq$ of $\Fq$.
    Then $\bar C$ is irreducible, and since $\bar f$ is continuous, we get that $\bar f(\bar C)$ is irreducible as well.
    Since $\bar f(\bar C)$ is the extension of $D$ to $\bar\Fq$, we conclude that $D$ is absolutely irreducible.
\end{proof}
The statement and proof of Lemma~\ref{lem:ggg} carry over verbatim if we replace $\Fq$ by an arbitrary perfect field.

The following results relate the existence of a principal polarisation on an abelian surface $A$ and of curves of genus $2$ on $A$.
\begin{lemma}\label{lem:Barth1.1}
    Let $C$ and $D$ be two $\Fq$-irreducible curves lying on an abelian surface $A$ such that $C.D=0$. Then, $C$ and $D$ are both elliptic curves lying on $A$.
\end{lemma}
\begin{proof}
    This is part of \cite[Lemma~1.1]{Barth} whose proof adapts to finite fields without any change.
\end{proof}

\begin{proposition}\label{prop:pp_iff_pa2}
    Let $A$ be a simple abelian surface over $\Fq$.
    Then the following statements are equivalent:
    \begin{enumerate}[(i)]
        \item\label{prop:pp_iff_pa2:pp} $A$ has a principal polarisation;
        \item\label{prop:pp_iff_pa2:pa2} $A$ contains an $\Fq$-irreducible curve of arithmetic genus $2$.
    \end{enumerate}
\end{proposition}
\begin{proof}
    Assume that $A$ has a principal polarisation.
    Since $\Fq$ is finite, there exists an ample invertible sheaf $\mathcal{L}\in\mathrm{Pic}(A)$ defined over $\Fq$ such that the principal polarisation is given by $\lambda_\mathcal{L}$ \cite[Remark 13.2]{CS86}.
    By \cite[p.~150]{M_AV70}, if we write $\mathcal{L}=\mathcal{L}(D)$ for some effective rational divisor $D$ then 
    \begin{equation}\label{deg_pol_self_intersection}
       \deg \lambda_\mathcal{L}= \chi(D)^2=(D^2/2)^2.
    \end{equation}
    Since $\deg \lambda_\mathcal{L}=1$, it follows that 
    $D^2=2$.
    Hence, if $D=E+F$ were reducible then, from the formula 
    \begin{equation}\label{eq:div_square} 
        D^2=E^2+F^2+2\cdot E.F,
    \end{equation}
    we would get that $D$ is the sum of two elliptic curves.
    Since $A$ is simple, it does not contain any elliptic curve.
    Hence, $D$ is an $\Fq$-irreducible curve of arithmetic genus $2$.

    Conversely, assume that $A$ contains an $\Fq$-irreducible curve $C$ of arithmetic genus $2$.
    By the adjunction formula, we get $C^2=2$.
    Lemma \ref{lem:Barth1.1} implies that $C$ intersects any other $\Fq$-irreducible curve on $A$ positively. Hence, the Nakai--Moishezon criterion \cite[Sec.~5, Thm.~1.10]{Hart} says that the divisor defined by $C$ is ample.
    So, the induced isogeny $A \to A^\vee$ is a polarisation.
    Therefore, $C$ defines a principal polarisation on $A$ by Equation~\eqref{deg_pol_self_intersection}.
\end{proof}

Part of the statement of the next theorem is proved in \cite{MYFE1}. However, as we will discuss in Remark \ref{rmk:gap}, there were some gaps that we fill in.
\begin{theorem}\label{thm:no_g1g2}
    Let $A$ be an abelian surface defined over $\Fq$ with
    Weil polynomial 
    \[ f(t) = t^4 +a t^3 +bt^2+qat+q^2. \]
    Then, the following statements are equivalent:
    \begin{enumerate}[label = (\roman*), ref=(\roman*)]
    \item\label{nocurves} $A$ does not contain absolutely irreducible curves of geometric genus $0, 1$ or $2$;
    \item\label{simple_nojac}   $A$ is simple and not isogenous to the Jacobian of an absolutely irreducible smooth genus $2$ curve;
    \item\label{nopp_wr} exactly one of the following statements holds:
        \begin{enumerate}[label = (\alph*), ref=\theenumi{}.(\alph*)]
            \item\label{not_isog_PP} $A$ is not isogenous to a principally polarised abelian surface, which is equivalent to have $a^2-b =q$, $b<0$ and all prime divisors of $b$ are congruent to $1 \bmod 3$;
            \item \label{weil_restr} $A$ is isogenous to a Weil restriction of an elliptic curve defined over the quadratic extension of $\Fq$ (hence $a=0$) and one of the following conditions holds:
            \begin{itemize}
                \item $b = 1-2q$;
                \item $p>2$ and $b=2-2q$;
                \item $p \equiv 11 \bmod 12$, $q$ is a square and $b=-q$;
                \item $p=3$, $q$ is a square and $b=-q$;
                \item $p=2$, $q$ is nonsquare and $b=-q$;
                \item $q=2$ and $b=-4$;
                \item $q=3$ and $b=-6$.
            \end{itemize}
        \end{enumerate}
    \end{enumerate}
\end{theorem}
\begin{proof}
    We start by proving \ref{nocurves}$\Rightarrow$\ref{simple_nojac}. By assumption, $A$ does not contain an elliptic curve, hence it is simple. Assume that $A$ is isogenous to the Jacobian of a smooth absolutely irreducible genus $2$ curve $C$.
    Since $C$ is canonically embedded into its Jacobian, Lemma~\ref{lem:ggg} would imply that $A$ contains an absolutely irreducible curve of geometric genus $\leq 2$, in contradiction with \ref{nocurves}.

    Now we show that \ref{simple_nojac}$\Rightarrow$\ref{nocurves}.
    First note that $A$ cannot contain an absolutely irreducible curve $D$ of geometric genus $0$.
    Indeed, the normalisation $\tilde D$ of $D$ is birationally equivalent to $\mathbb{P}^1$ and, by \cite[Cor.~3.8]{CS86}, the only rational maps from $\mathbb{P}^1$ to $A$ are the constant maps.
    The abelian variety $A$ cannot contain an absolutely irreducible curve of geometric genus $1$ as well.
    Indeed, the normalisation of such a curve would be an elliptic curve (see for example \cite[Prop.~1]{Haloui17}) and hence an isogeny factor of $A$. 
    Similarly, $A$ cannot contain an absolutely irreducible curve $D$ of geometric genus $2$ because the induced map $\tilde D\to A$ would give a morphism $\Jac(\tilde D) \to A$, which is an isogeny since $A$ is simple.

    We now focus on \ref{simple_nojac}$\Rightarrow$\ref{nopp_wr}.
    We distinguish two cases.
    Assume first that $A$ is not isogenous to a principally polarised abelian surface.
    This corresponds to \ref{not_isog_PP}. 
    The equivalence in terms of the coefficients of the Weil polynomial follows from \cite[Thm.~1]{HMNR08}.
    Assume now that $A$ is isogenous to a principally polarised abelian surface.
    By a classification theorem due to Weil (see for instance \cite[Thm.~1.3]{HNR09}), a principally polarised abelian surface defined over $\Fq$ is exactly one of the following: a product of two elliptic curves defined over $\Fq$ with the product polarisation; the Jacobian of an absolutely irreducible genus $2$ curve defined over $\Fq$ with the canonical polarisation; the Weil restriction of an elliptic curve defined over the quadratic extension of $\Fq$ with the induced polarisation. 
    Since we are assuming \ref{simple_nojac}, we exclude the first two cases.
    So, $A$ is the Weil restriction of an elliptic curve defined over the quadratic extension of $\Fq$.
    Finally, the characterisation of the coefficients of the Weil polynomial follows from \cite[Prop.~4.3]{MYFE1} and the beginning of its proof.
    Hence, we are in case \ref{weil_restr}.

    We conclude by showing that \ref{nopp_wr}$\Rightarrow$\ref{simple_nojac}.
    For $A$ satisfying \ref{not_isog_PP} the implication is clear.
    If $A$ satisfies \ref{weil_restr} the result follows from \cite[Thm.~1.2-(2), Table~1.2]{HNR09}.
\end{proof}

\begin{remark}\label{rmk:gap}
    In \cite[Lemma~4.1]{MYFE1}, it is stated that for an abelian surface $A$ defined over $\Fq$ we have that the following statements are equivalent:
    \begin{enumerate}[(I)]
        \item\label{item_a} $A$ is simple and not isogenous to a Jacobian surface, that is, the Jacobian of an absolutely irreducible smooth genus $2$ curve;
        \item\label{item_b} $A$ does not contain absolutely irreducible curves of arithmetic genus $0$, $1$ or $2$;
        \item\label{item_c} $A$ does not contain absolutely irreducible smooth curves of genus $0$, $1$ or $2$.
    \end{enumerate}
    Observe that \ref{item_a}$=$\ref{simple_nojac} and that \ref{nocurves}$\Rightarrow$\ref{item_b}.
    Hence, the implication \ref{simple_nojac}$\Rightarrow$\ref{nocurves} of Theorem \ref{thm:no_g1g2} is a stronger statement than the implication \ref{item_a}$\Rightarrow$\ref{item_b}.
    However, while our reverse implication \ref{nocurves}$\Rightarrow$\ref{simple_nojac} holds true, the  implication \ref{item_b}$\Rightarrow$\ref{item_a} claimed in \cite[Lemma~4.1]{MYFE1} is false. 
    Indeed, consider a simple isogeny class containing a Jacobian surface and non-principally polarisable abelian surfaces. 
    Then, any 
    abelian surface admitting no principal polarisation inside such an isogeny class gives an example of an abelian surface isogenous to a Jacobian and containing no $\Fq$-irreducible curves of arithmetic genus $2$ by Proposition~\ref{prop:pp_iff_pa2} above.
    An example of such a class is given in Example \ref{ex:controesempio}.

    Finally, let us remark that the implication \ref{item_c}$\Rightarrow$\ref{item_b} of \cite[Lemma~4.1]{MYFE1} still holds true, but a new proof is necessary, that we offer here.
    Let us suppose that $A$ contains an absolutely irreducible curve of arithmetic genus $p_a=1$.
    Then the geometric genus is necessarily $1$ as well, since $g$ cannot be $0$ and $g\leq p_a$, thus the curve is smooth, a contradiction.
    Suppose now that  $A$ contains an absolutely irreducible curve of arithmetic genus $p_a=2$.
    If the geometric genus is $1$ then we know by \cite[Prop.~1]{Haloui17} that the curve is smooth and has a structure of an elliptic curve, which leads to a contradiction.
    If the geometric genus is $2$ then the curve is smooth of genus $2$, leading again to a contradiction.  
\end{remark}

\begin{remark}\label{rms:genus_2_invariant}
    We highlight that there is a crucial difference between considering the arithmetic genus and the geometric genus in our characterisation. On the one hand, Lemma \ref{lem:ggg} implies that if an abelian surface $A$ doesn't contain absolutely irreducible curves of geometric genus $\leq 3$ then every abelian variety isogenous to $A$ doesn't contain absolutely irreducible curves of geometric genus $\leq 3$. 
    On the other hand, this is not the case for absolutely irreducible curves of arithmetic genus $\leq 2$ as we have seen in Remark~\ref{rmk:gap}, or for absolutely irreducible curves of arithmetic genus $\leq 3$ as we will see using Theorem~\ref{thm:pol_deg_4_iff_pa3}.
\end{remark}

\begin{remark}\label{rmk:good_to_know_Weil_res} 
  Note that one can distinguish isogeny classes of abelian surfaces admitting no principal polarisation and simple isogeny classes of Weil restrictions by their splitting behavior.    From \cite[p.~122]{HMNR08}, we know that an abelian surface $A$ defined over $\Fq$ which is in a not principally polarisable isogeny class splits over the cubic extension of the base field, that is $A\otimes_{\Fq}\F_{q^3}$ is not simple.
    If $A$ is a simple abelian surface over $\F_q$ then $A$ is isogenous to a Weil restriction of an elliptic curve defined over $\F_{q^2}$ if and only if $A\otimes_{\Fq}\F_{q^2}$ is not simple; see \cite[Lemma~4]{HMNR08}.
\end{remark}

The rest of the section is devoted to prove various results about the Weil polynomials from Theorem \ref{thm:no_g1g2}.

\begin{lemma}\label{lemma:f_irred_almost}
    Let $f(t)$ be a Weil polynomial of an abelian surface satisfying one of the equivalent conditions of Theorem \ref{thm:no_g1g2}.
    Then either $f(t) = (t^2-q)^2$ with $q\in \{2,3\}$, or $f(t)$ is irreducible over $\Q[x]$.
\end{lemma}
\begin{proof}
    Since $f(t)$ is the Weil polynomial of a simple abelian surface then either it is irreducible or it is the square of a quadratic polynomial.
    Assume that $f(t)$ is reducible.
    Then either $f(t)=(t^2-q)^2$ or $f(t)=(t^2-\beta t+q)^2$ for some $\beta\in\Z$.
    If we are the in the former case, then by comparing with Theorem~\ref{thm:no_g1g2}, we see that $q\in\{2,3\}$, and we are done.
    So assume that we are in the latter case.
    Then $a=-2\beta$ and $b=2q+\beta^2$.
    Observe that $b>0$ which implies that $f(t)$ does not satisfy condition \ref{not_isog_PP}.
    So we must be in case \ref{weil_restr}.
    Hence, $a=0$ which implies that $\beta=0$ and $b=2q$, giving a contradiction with the constraints on $b$ from \ref{weil_restr}.
    Therefore, $f(t)$ is irreducible.
\end{proof}

In view of Lemma~\ref{lemma:f_irred_almost}, the set of Weil polynomials described in Theorem~\ref{thm:no_g1g2} can be partitioned as follows.

\begin{definition}\label{df:partition}
    Let $\PirrA$ (resp.~$\PirrB$) be the set of Weil polynomials described in Theorem~\ref{thm:no_g1g2}.\ref{not_isog_PP} (resp.~Theorem~\ref{thm:no_g1g2}.\ref{weil_restr}) which are irreducible over $\Q[x]$.
    The set of Weil polynomials described in Theorem~\ref{thm:no_g1g2} is partitioned as follows:
    \[ \PirrA\sqcup\PirrB \sqcup  \{ (t^2-2)^2 , (t^2-3)^2 \}.\]
\end{definition}

After having introduced this notation, we are ready to complete the proof of Main Theorem~\ref{mainthm0}.
\begin{corollary}\label{cor:mainthm0}
    Let $f(t)$ be a Weil polynomial defining a simple isogeny class $\cC$.
    \begin{enumerate}[(i)]
        \item \label{cor:mainthm0:Pnpp} $f(t)$ is in $\PirrA$ if and only if no $A$ in $\cC$ contains an $\Fq$-irreducible curve with arithmetic genus $\leq 2$.
        \item \label{cor:mainthm0:PWres} $f(t)$ is in $\{ (t^2-2)^2 , (t^2-3)^2 \} \sqcup\PirrB$ if and only if no $A$ in $\cC$ contains an absolutely irreducible curve with geometric genus $\leq 2$ and there exists a $B$ in $\cC$ which contains an $\Fq$-irreducible curve with arithmetic genus $2$.
    \end{enumerate}
\end{corollary}
\begin{proof}
    Since $\cC$ is simple, no $A$ in $\cC$ can contain an $\Fq$-irreducible curve of arithmetic genus $1$.
    Indeed, any such curve would have geometric genus $1$, that is, it would be an elliptic curve which is implies that $A$ is isogenous to the product of two elliptic curves.
    Hence, Part~\ref{cor:mainthm0:Pnpp} follows immediately from Proposition~\ref{prop:pp_iff_pa2}.

    Now, we prove Part~\ref{cor:mainthm0:PWres}. 
    Assume first that $f(t)$ is in $\{ (t^2-2)^2 , (t^2-3)^2 \} \sqcup\PirrB$ and let $A$ be an abelian surface in $\cC$.
    Then, $A$ does not contain absolutely irreducible curves of geometric genus $\leq 2$ by Theorem~\ref{thm:no_g1g2}.
    Moreover, there exists an abelian surface in $\cC$ which is principally polarised and hence contains an $\Fq$-irreducible curve $C$ of arithmetic genus $2$ by Proposition~\ref{prop:pp_iff_pa2}.
    Note that this curve $C$ has geometric genus $\leq 2$, so it is not absolutely irreducible because of the assumptions on $A$.
    
    For the converse implication, we first observe that $f(t)$ is in $\{ (t^2-2)^2 , (t^2-3)^2 \} \sqcup\PirrB \sqcup\PirrA$ by Theorem~\ref{thm:no_g1g2}. 
    By Proposition~\ref{prop:pp_iff_pa2}, we have that $B$ is principally polarised and thus $f(t) \not\in \PirrA$.
\end{proof}

Recall that if $A$ is an abelian variety over $\Fq$ whose Weil polynomial $f(t)$ has complex roots $\alpha_1\ldots,\alpha_{2g}$, then the Weil polynomial of the extension of scalars $A_{\F_{q^n}}=A \otimes_{\Fq} \F_{q^n}$ has complex roots $\alpha_1^n,\ldots,\alpha_{2g}^n$.
This observation, combined with the next result, gives us an effective way to test whether an irreducible Weil polynomial belongs to $\PirrB$.

\begin{lemma}\label{lemma:PirrB_char}
    Let $A$ be an abelian surface over $\Fq$ with Weil polynomial $f(t)=t^4+at^3+bt^2+qat+q^2$.
    Then $f(t) \in \PirrB$ if and only if $f(t)$ is irreducible, the Weil polynomial
    of $A_{\F_{q^2}}$ is not irreducible, and $b$ is as in Theorem~\ref{thm:no_g1g2}.\ref{weil_restr}.
\end{lemma}
\begin{proof}
    It follows from Theorem \ref{thm:no_g1g2}.\ref{weil_restr} and Remark~\ref{rmk:good_to_know_Weil_res}.
\end{proof}

For the rest of the section we focus on Weil polynomials belonging to $\PirrA\sqcup\PirrB$.

\begin{lemma}\label{lem:Pirrordinary}
    Let $A$ be an abelian surface over $\Fq$ with Weil polynomial $f(t)=t^4+at^3+bt^2+qat+q^2$ in $\PirrA\sqcup\PirrB$. Then $A$ is either ordinary or supersingular. 
    If $f(t)\in\PirrA$ then $A$ is ordinary if and only if $a\neq 0$. If $f(t)\in\PirrB$, then $A$ is ordinary if and only if $b = 1-2q$ or, $b=2-2q$ and $p>2$.
\end{lemma}
\begin{proof}
    Assume that $f(t)\in\PirrA$.  Note that $A$ cannot have $p$-rank $1$ since in that case it would be principally polarisable by \cite[Thm.~4.3]{MN02}\footnote{The authors of \cite{MN02} call the neither ordinary nor supersingular case ``mixed''.}. 
    If $a=0$ then we have $b=a^2-q=-q$, and therefore $A$ is supersingular. 
    Suppose $a\neq 0$. By \cite[Thm.~2]{HMNR08}, the abelian variety $A$ cannot be supersingular. Hence, $A$ is ordinary.  
    Now, assume that $f(t)\in\PirrB$. 
    Obviously if $b = 1-2q$, or $b=2-2q$ and $p>2$ then $A$ is ordinary. 
    By Theorem  \ref{thm:no_g1g2}.\ref{weil_restr}, those are the only possible values for $b$ to have ordinarity.
    To conclude the proof, we need to show that if $A$ is non-ordinary then it is supersingular.
    Say that $A$ is the Weil restriction of an elliptic curve $E$ over $\F_{q^2}$. 
    If $A$ is non-ordinary, then $E$ is supersingular, and hence $A$ is also supersingular.
\end{proof}

\begin{proposition}\label{prop:galois}
    Let $f(t)=t^4+at^3+bt^2+qat+q^2$ be a Weil polynomial in $\PirrA\sqcup\PirrB$.
    Then, the number field $K=\Q[t]/f(t)$ is Galois.
\end{proposition}

\begin{proof}
    Let $\alpha, q/\alpha, \beta, q/\beta$ be the complex roots of $f(t)$.
    Then $-a=\alpha+q/\alpha+\beta+q/\beta$.
    If $a=0$ then either $\alpha=-\beta$ or $\alpha=-q/\beta$.
    In both cases $\beta\in\Q(\alpha)$, so the extension is normal and $K$ is Galois. 
    If $a\neq 0$ then $f(t)$ is in $\PirrA$ and $A$ is ordinary by Lemma \ref{lem:Pirrordinary}.
    Hence $K$ is Galois by \cite[Lemma 12.1]{Howe95}.
\end{proof}

\begin{lemma}\label{lemma:splitting_mod_2}
    Every polynomial $f(t) \in \PirrA\sqcup\PirrB$ is congruent to the product of two quadratic polynomials modulo $2$.
\end{lemma}
\begin{proof}
    If $f(t) \in \PirrB$ then it is an easy calculation.
    Assume now that $f(t) \in \PirrA$.
    As usual, write $f(t) = t^4+at^3+bt^2+aqt+q^2$.
    The statement follows from two remarks.
    Firstly, since all prime divisors of $b$ are $\equiv 1 \mod 3$, $b$ must be odd.
    Secondly, since $b=a^2-q$, if $q$ is even then we must have $a^2$ (and thus $a$) odd, while if $q$ is odd we must have $a^2$ (and thus $a$) even.
    Therefore, an easy calculation shows that if $q$ is even then $f(t)\equiv t^2(t^2+t+1) \mod 2$, and that if $q$ is odd then $f(t)\equiv (t^2+t+1)^2 \mod 2$.
\end{proof}

\section{Abelian surfaces containing genus \texorpdfstring{$3$}{3} curves}\label{sec:g3onsurf}
The key question of the paper is characterising abelian surfaces which do not contain genus $3$ curves.
After a technical lemma, we prove Main Theorem~\ref{mainthm1}, stated below as Theorem~\ref{thm:pol_deg_4_iff_pa3}.
\begin{lemma}\label{lem:div4}
    Let $A$ be a simple abelian surface defined over $\Fq$. Let $D$ be an effective rational divisor on $A$ such that $D^2=4$. Then, $D$ is an $\Fq$-irreducible curve of arithmetic genus $3$. More precisely, $D$ is one of the following:
    \begin{itemize}
    \item an absolutely irreducible smooth curve of genus $3$;
    \item an absolutely irreducible curve of geometric genus $2$ with a double point;
    \item an $\Fq$-irreducible curve of arithmetic genus $3$ which is not absolutely irreducible.
    \end{itemize}
\end{lemma}
\begin{proof}
    By the adjunction formula, for an $\Fq$-irreducible curve $D$ over $A$ we have $D^2=2p_a(D)-2$.
    If $D$ is an $\Fq$-irreducible curve then it necessarily has arithmetic genus $3$,
    and it must be one of the three possibilities listed in the statement.
    Hence, to conclude the proof it suffices to show that $D$ is an $\Fq$-irreducible curve.

    If $D=E+F$ were reducible, using Equation~\eqref{eq:div_square}
    and the assumption that $D^2=4$, we would get that it has at most $2$ irreducible components having arithmetic genus strictly lower than $3$.
    Since $A$ is simple, the components $E$ and $F$ of $D$ cannot be elliptic curves.
    Hence, they necessarily have arithmetic genus $2$.
    So, again by the adjunction formula, we get that $E^2=F^2=2$.
    Moreover, we have $E.F>0$ by Lemma \ref{lem:Barth1.1}.
    Combining these with Equation~\eqref{eq:div_square}, we get $D^2>4$ which is a contradiction.
    It follows that $D$ must be an $\F_q$-irreducible curve, thus concluding the proof.
 \end{proof}

\begin{theorem}\label{thm:pol_deg_4_iff_pa3}
    Let $A$ be a simple abelian surface defined over $\Fq$. Then, the following statements are equivalent:
    \begin{enumerate}[(i)]
        \item\label{g3_two} $A$ has a polarisation of degree $4$;
        \item\label{g3_three} $A$ contains an $\Fq$-irreducible curve of arithmetic genus $3$.
    \end{enumerate}
\end{theorem}
\begin{proof}
    We start by proving \ref{g3_two}$\Rightarrow$\ref{g3_three}.
    Following the reasoning in the proof of \cref{prop:pp_iff_pa2}, we see that the polarisation of degree $4$ is given by an effective rational divisor $D$ on $A$ with self-intersection $D^2=4$.
    By Lemma \ref{lem:div4} we conclude.

    To prove \ref{g3_three}$\Rightarrow$\ref{g3_two}, let $C$ be an $\Fq$-irreducible curve of arithmetic genus $3$ lying over $A$. 
    As in the proof of Proposition~\ref{prop:pp_iff_pa2}, the adjunction formula, Lemma \ref{lem:Barth1.1} and the Nakai--Moishezon criterion imply that $C$ is an ample divisor, which in turn defines a polarisation of degree $4$ by Equation~\eqref{deg_pol_self_intersection}.
\end{proof}

\begin{remark}
    If $A$ does not contain an absolutely irreducible curve of geometric genus $2$, then, from Lemma \ref{lem:div4}, we deduce that the genus $3$ $\Fq$-irreducible curve in the statement of Theorem \ref{thm:pol_deg_4_iff_pa3} is either an absolutely irreducible smooth curve of genus $3$ or an $\Fq$-irreducible curve of arithmetic genus $3$ which is not absolutely irreducible.
\end{remark}

From now on, our main focus will be to characterise the abelian surfaces with a polarisation of degree $4$ among those with Weil polynomial in $\PirrA\sqcup\PirrB \sqcup  \{ (t^2-2)^2 , (t^2-3)^2 \}$, which as we know by Theorem \ref{thm:no_g1g2} do not contain curves of geometric (hence arithmetic) genus up to $2$.

We start by analysing the two cases of $(t^2-2)^2 $ and $(t^2-3)^2$ in Proposition \ref{prop:q2q3}.
In the following sections, we will apply Theorem~\ref{thm:pol_deg_4_iff_pa3} to characterise which isogeny classes with Weil polynomial in $\PirrA\sqcup\PirrB$ contain an abelian surface containing a curve of genus $3$.

\begin{proposition}\label{prop:q2q3}
    Let $\cC$ (resp.~$\cC'$) be the isogeny class determined by $(t^2-2)^2$ (resp.~$(t^2-3)^2$).
    Then 
    \begin{itemize}
        \item no abelian surface in $\cC$ contains an absolutely irreducible curve of geometric (and thus arithmetic) genus $3$.
        \item there exists an abelian surface in $\cC'$ containing the absolutely irreducible smooth curve of genus $3$ defined by $y^4+2x^3z+xz^3$.
    \end{itemize}
\end{proposition}
\begin{proof}
    If an abelian surface $A$ in $\cC$ contains an absolutely irreducible curve $C$ of geometric genus $3$, then $A$ is an isogeny factor of the Jacobian $\Jac(\tilde C)$ of the normalisation $\tilde C$ of $C$ by \cite[Prop.~2]{Haloui17}\footnote{The author requires the curve to have a rational point on the smooth curve to guarantee the existence of an embedding into its Jacobian. However, the existence of a divisor of degree $1$ is sufficient. A curve defined over a finite field always admits such a divisor.\label{footnote}}.
    Hence, the Weil polynomial of $\Jac(\tilde C)$ must be of the form $(t^2-2)^2 f_E(t)$, where $f_E(t)$ is the Weil polynomial of an elliptic curve over $\F_2$.

    Note that there are only $5$ distinct isogeny classes of elliptic curves over $\F_2$. 
    A search in the database of isogeny classes of abelian threefolds defined over $\F_2$ in the LMFDB \cite{lmfdb} returns the isogeny classes 
    \href{http://www.lmfdb.org/Variety/Abelian/Fq/3/2/ac_ac_i}{3.2.ac\_ac\_i}, 
    \href{http://www.lmfdb.org/Variety/Abelian/Fq/3/2/ab_ac_e}{3.2.ab\_ac\_ e}, 
    \href{http://www.lmfdb.org/Variety/Abelian/Fq/3/2/a_ac_a}{3.2.a\_ac\_a}, 
    \href{http://www.lmfdb.org/Variety/Abelian/Fq/3/2/b_ac_ae}{3.2.b\_ac\_ae} and 
    \href{http://www.lmfdb.org/Variety/Abelian/Fq/3/2/c_ac_ai}{3.2.c\_ac\_ai}.
    None of them contains a Jacobian, which implies that $A$ does not contain an absolutely irreducible curve of geometric genus $3$.
    Hence, $A$ does not contain an absolutely irreducible curve of arithmetic genus $3$, as well.

    We now apply a similar reasoning to $\cC'$.
    By searching in the LMFDB, we find $7$ isogeny classes of the form $(t^2-3)^2 f_{E'}(t)$ for some elliptic curve $E'/\F_3$.
    Moreover, $6$ of these $7$ isogeny classes do not contain a Jacobian.
    See 
    \href{http://www.lmfdb.org/Variety/Abelian/Fq/3/3/ad_ad_s}{3.3.ad\_ad\_s},
    \href{http://www.lmfdb.org/Variety/Abelian/Fq/3/3/ac_ad_m}{3.3.ac\_ad\_m},
    \href{http://www.lmfdb.org/Variety/Abelian/Fq/3/3/ab_ad_g}{3.3.ab\_ad\_g},
    \href{http://www.lmfdb.org/Variety/Abelian/Fq/3/3/b_ad_ag}{3.3.b\_ad\_ag},
    \href{http://www.lmfdb.org/Variety/Abelian/Fq/3/3/c_ad_am}{3.3.c\_ad\_am} and
    \href{http://www.lmfdb.org/Variety/Abelian/Fq/3/3/d_ad_as}{3.3.d\_ad\_as}.
    The remaining one has label
    \href{http://www.lmfdb.org/Variety/Abelian/Fq/3/3/a_ad_a}{3.3.a\_ad\_a} and
    Weil polynomial $(t^2-3)^2(t^2+3)$.
    It contains only the Jacobian surface $J$ of the curve $C:y^4+2x^3z+xz^3$.
    The degree $2$ 
    involution sending $(x:y:z)\mapsto (x:-y:z)$  defines a double cover $\pi:C\to F$ where $F$ is the elliptic curve given by $y^2z+2x^3+xz^2$.
    Hence, $C$ is bielliptic, and
    $J/\pi^* F \in \cC'$. 
    The branch points of $\pi$ are the images of the fixed points of the involution, that is, the points on $C$ with $y=0$.
    We get the points $(0:0:1)$, $(1:0:1)$ and $(2:0:1)$ together with the point at infinity $(1:0:0)$. 
    We now follow Barth's proof \cite[Prop. 1.8]{Barth} (the forward direction of the implication),
    where the author defines a map $\varphi: C \hookrightarrow J \to J/\pi^* F$ using any of the branch points of $\pi$.
    A careful analysis of the construction shows that, since the branch point is defined over $\F_3$, the same holds for $\varphi$.
    Barth's argument to show that $\varphi$ is an embedding carries over to finite fields without any modification.
    This concludes the proof.
   \end{proof}

Having this characterisation at hand, we can easily derive the Corollary in the introduction.
\begin{corollary}\label{cor:consequence}
    Let $f(t)$ be the Weil polynomial of a simple isogeny class $\cC$ of abelian surfaces over $\Fq$.
    \begin{enumerate}[(i)]
        \item \label{cor:consequence:1} 
            Assume that $f(t)\in\PirrA\sqcup\PirrB$ and let $A$ be an abelian surface in $\cC$.
            If $A$ contains an absolutely irreducible curve of arithmetic genus $\leq 3$ then $A$ admits a polarisation of degree $4$.
        \item \label{cor:consequence:2}
            No abelian surface in $\cC$ contains an $\Fq$-irreducible curve of arithmetic genus $\leq 3$ if and only if $f(t)$ is in $\PirrA$ and no abelian surface in $\cC$ admits a polarisations of degree~$4$.
    \end{enumerate}   
\end{corollary}
\begin{proof}
    We start by proving \ref{cor:consequence:1}.
    Let $C$ be an absolutely irreducible curve of arithmetic genus $\leq 3$ on $A$.
    By the assumption on $f(t)$, Theorem~\ref{thm:no_g1g2} implies that $p_a(C)=3$. 
    Therefore, $A$ admits a polarisation of degree $4$ by Theorem~\ref{thm:pol_deg_4_iff_pa3}.

    We now continue with \ref{cor:consequence:2}.
    Since $\cC$ is simple, no $A$ in $\cC$ contains an $\Fq$-irreducible curve of arithmetic or geometric genus $1$, because such a curve would be an elliptic curve and hence an isogeny factor of $\cC$.
    Then we conclude the proof by observing that Proposition~\ref{prop:pp_iff_pa2} and Theorem~\ref{thm:pol_deg_4_iff_pa3} imply that an abelian surface $A$ in $\cC$ contains an $\Fq$-irreducible curve of arithmetic genus $\leq 3$ if and only if $A$ admits a polarisation which is of degree $1$ or $4$.
\end{proof}

\section{Factorisation of \texorpdfstring{$2$}{2}}
\label{sec:fact_of_2}

To prove Main Theorem~\ref{mainthm2}, we first need some technical results concerning the factorisation of $2$ in the extension $K = \Q[t]/f(t)$ where $f(t)$ is a Weil polynomial in $\PirrA\sqcup\PirrB$.
Note that $K$ is a CM-field, that is, a totally imaginary quadratic extension of a totally real field $K^+$.
Concretely, if we denote by $\pi$ the class of $t$ in $K$, then $K^+=\Q(\pi+q/\pi)$.
Moreover, if we write 
\[ f(t) = t^4+at^3+bt^2+aqt+q^2 \]
then the minimal polynomial of $\pi+q/\pi$ is 
\[ f^+(t) = t^2 + at + (b-2q). \]
We will denote by $\cO$ (resp.~$\cO^+$) the ring of integers of $K$ (resp.~$K^+$).

In this section, we study the factorisation of $2$ in $K^+$ and $K$. 
We shall use these technical results in the following sections to study when an isogeny class contains an abelian surface admitting a polarisation of degree $4$.

\begin{lemma}\label{lem:ram_ram}
    Assume that $f(t) \in \PirrB$. 
    Then $2$ ramifies in $K^+$.
\end{lemma}
\begin{proof}
    Since $f(t)\in \PirrB$ then $a=0$.
    Hence, $\Delta_{f^+} = 4(2q-b)$.
    Let $d$ be a positive squarefree integer such that $2q-b = c^2 d$ with $c\in \Z$.
    Note that if $c$ is odd then $2q-b \equiv d \bmod 4$.     Recall that $2$ ramifies in $K^+=\Q(\sqrt{d})$ if and only if $d=1$ or $d\equiv 2,3\bmod 4$.
    We now analyse all the possible cases, which are described in Theorem~\ref{thm:no_g1g2}.\ref{weil_restr}.
    \begin{itemize}
        \item If $b = 1-2q$ then $2q-b = 4q-1$. So $c$ is odd. Hence $d \equiv 3 \bmod 4$.
        \item If $b = 2-2q$ then $2q-b = 4q-2$. So $c$ is odd. Hence $d \equiv 2 \bmod 4$.
        \item If $b=-q$ and $q$ is a square then $d = 3$.
        \item If $b=-q$, $p=2$ and $q$ is not a square then $d=6 \equiv 2\bmod 4$.
    \end{itemize}
\end{proof}

\begin{lemma}\label{lem:ram_tot_ram}
    Assume that $f(t) \in \PirrB$.
    Let $\frm$ be the unique maximal ideal of $\cO^+$ above $2$ (cf.~Lemma~\ref{lem:ram_ram}).
    Then the following are equivalent:
    \begin{enumerate}[(i)]
        \item \label{lem:ram_tot_ram:KKp_ram} $K/K^+$ is ramified.
        \item \label{lem:ram_tot_ram:q_ram} $\frm$ is the unique maximal ideal of $\cO^+$ that ramifies in $K$.
        \item \label{lem:ram_tot_ram:2_tot_ram} $2$ is totally ramified in $K$.
        \item \label{lem:ram_tot_ram:b} $b=2-2q$ with $q$ odd.
    \end{enumerate}
    In particular, if any of the above equivalent conditions holds then $f(t)$ is ordinary.
\end{lemma}
\begin{proof}
    Clearly, \ref{lem:ram_tot_ram:q_ram}$\Rightarrow$\ref{lem:ram_tot_ram:KKp_ram} and \ref{lem:ram_tot_ram:2_tot_ram}$\Rightarrow$\ref{lem:ram_tot_ram:KKp_ram}.
    Before proving the remaining implications, we pause to prove a general claim.
    Set $N= N_{K^+/\Q}(\Delta_{K/K^+})$.
    We have
    \begin{align*}
        \Delta_f & = 16q^2(4q^2-b^2)^2=[\cO:\Z[\pi]]^2\Delta_K = [\cO:R]^2q^2\Delta_K,\\
        \Delta_{f^+} & = 4(2q-b) = [\cO^+:R^+]^2\Delta_{K^+}, \text{ and}\\
        \Delta_K & = \Delta_{K^+}^2 N.
    \end{align*}
    Combining these relations, we obtain
    \[ N \cdot \left( \frac{[\cO:R]}{[\cO^+:R^+]^2} \right)^2 = (2q+b)^2. \]
    Note that $R=R^+[\pi]$ (resp.~$\cO^+[\pi]$) is locally free of rank $2$ over $R^+$ (resp.~$\cO^+$).
    Hence $[\cO^+:R^+]^2 = [\cO^+[\pi]:R]$ divides $[\cO:R]$.
    Therefore, 
    \begin{equation}\label{eq:Ndiv} 
        N\text{ divides }(2q+b)^2.
    \end{equation}

    Now, we show that \ref{lem:ram_tot_ram:KKp_ram}$\Rightarrow$\ref{lem:ram_tot_ram:b} and \ref{lem:ram_tot_ram:KKp_ram}$\Rightarrow$\ref{lem:ram_tot_ram:q_ram}.
    So we assume that $N\neq 1$.

    Assume first that $f$ is non-ordinary.
    Then $f(t)=t^4-qt^2+q^2$ by Theorem~\ref{thm:no_g1g2} and Lemma~\ref{lem:Pirrordinary}.
    Note that $\zeta_3 \mapsto -\pi^2/q$ gives an embedding of $\Q(\zeta_3)$ in $K$ (cf.~\cite[Sec.~4]{HMNR08}).
    Moreover, as we have shown in the proof of Lemma~\ref{lem:ram_ram}, we have that $K^+=\Q(\sqrt{3})$ or $K^+=\Q(\sqrt{6})$.
    In the first case, we have $K=\Q(\zeta_3,\sqrt{3})$ and one computes that $\Delta_{K^+}=12$ and $\Delta_K=12^2$.
    In the second case, we have $K=\Q(\zeta_3,\sqrt{6})$ and one computes that $\Delta_{K^+}=24$ and $\Delta_K=24^2$.
    In both cases $K/K^+$ is unramified.

    Hence, $f(t)$ must be ordinary.
    Then, by Lemma~\ref{lem:Pirrordinary}, either $b=1-2q$ or, $b=2-2q$ and $q$ is odd.
    In the first case, Equation~\eqref{eq:Ndiv} implies that $N=1$, which cannot happen by assumption.
    So we must be in the second case, that is, \ref{lem:ram_tot_ram:b} holds.
    Equation~\eqref{eq:Ndiv} implies that if $b=2-2q$ (with $q$ odd) then $N \in \{1,2,4\}$, and in fact $N \in \{2,4\}$ since we are assuming that $K/K^+$ is ramified.
    Since $\frm$ is the unique maximal ideal of $\cO^+$ above $2$ by Lemma~\ref{lem:ram_ram}, we obtain that $\frm$ ramifies in $K$, that is, \ref{lem:ram_tot_ram:q_ram} holds.

    To conclude, we prove that \ref{lem:ram_tot_ram:b}$\Rightarrow$\ref{lem:ram_tot_ram:2_tot_ram}. 
    Since $b=2-2q$ with $q$ odd, we have that $f(t) \equiv (t+1)^4 \bmod 2$.
    By the Kummer-Dedekind theorem \cite[Thm.~8.2]{Stevenhagen08}, the order $\Z[\pi]$ in $K$ has a unique maximal ideal above $2$ which is regular since the remainder of the division of $f(t)$ by $t+1$ is $(q-1)^2+2 \equiv 2 \bmod 2^2$, because $q$ is odd.
    This implies that $\cO$ has a unique maximal ideal above $2$ with residue field $\F_2$, that is, $2$ is totally ramified in $K$.
\end{proof}

\begin{lemma}\label{lem:if_ram_then_B}
    Let $f(t)\in \PirrA\sqcup\PirrB$.
    If $K/K^+$ is ramified, or if $K/K^+$ is unramified and there exists a maximal ideal of $\cO^+$ that divides $(\pi-q/\pi)$ which stays inert in $K$, then $f\in \PirrB$.
    Moreover, in the latter case, $f(t)$ is non-ordinary.
\end{lemma}
\begin{proof}
    If $K/K^+$ is ramified or if there exists a prime in $K^+$ that divides $(\pi-q/ \pi)$ and which stays inert in $K$, then, by \cite[Thm.~1.1]{Howe96}, there is a principally polarised abelian surface in the isogeny class associated to $f(t)$.
    Hence $f\in \PirrB$ by Theorem~\ref{thm:no_g1g2}.

    To show the last statement, assume that $K/K^+$ is unramified and that there exists a maximal ideal $\frn$ of $\cO^+$ dividing $(\pi-q/\pi)$ which stays inert in $\cO$.
    Write $f(t) = t^4 + bt^2 +q^2$, as usual.
    Since the minimal polynomial of $\pi+q/\pi$ is $f^+(t) = t^2 + b -2q$, we have that $(\pi+q/\pi)^2 + b -2q =0$, which combined with $(\pi-q/\pi)^2 = \pi^2 +(q/\pi)^2 -2q$ gives 
    \[ (\pi-q/\pi)^2 = -b - 2q. \]
    Assume by contradiction that $f(t)$ is ordinary.
    Then, by Lemma~\ref{lem:Pirrordinary}, either $b=1-2q$, or $b=2-2q$ with $q$ odd.
    If $b=1-2q$ then $(\pi-q/\pi)^2 = -1$ which implies that $1 \in \frn^2\cO$. 
    This is not possible since $\frn\cO$ is a maximal ideal of $\cO$ by assumption.
    If $b=2-2q$ with $q$ odd, then $2 \in \frn^2\cO$.
    By Lemma~\ref{lem:ram_ram}, it follows that $\frn$ is the unique maximal ideal of $\cO^+$ above $2$.
    By Lemma~\ref{lem:ram_tot_ram}, $\frn$ is ramified in $K$, leading to a contradiction also in this case.
    Therefore $f(t)$ is non-ordinary.
\end{proof}

\begin{lemma}\label{lemma:MinK}
    Assume that $f(t) \in \PirrA$.
    Then $K$ contains $M=\Q(\zeta_3)$.
\end{lemma}
\begin{proof}
    If $f(t)$ is ordinary then \cite[proof of Lemma~12.2]{Howe95} shows that $K$ contains $M=\Q(\zeta_3)$. 
    If $f(t)$ is not ordinary, then \cite[Thm.~2]{HMNR08} tells us that 
    \[ f(t) = t^4 -qt^2 + q^2. \]
    Then $\zeta_3 \mapsto -\pi^2/q$ gives an embedding of $M$ in $K$ (cf.~\cite[Sec.~4]{HMNR08}).
\end{proof}

\begin{proposition}\label{prop:fact_2}
    Let $f(t) \in \PirrA\sqcup\PirrB$.
    The possible factorisations of $2$ in $K$ are as follows and all of them do occur.
        \begin{enumerate}[(i)]
        \item\label{2inPA} If $f(t) \in\PirrA$ then
        \begin{itemize}
        \item if $2$ is inert in $K^+$ then $2\cO = \frM_1\frM_2$ with $\cO/\frM_1 \simeq \cO/\frM_2 \simeq \F_4$ and $\frM_1 = \overline{\frM}_2$;
        \item if  $2$ is split in $K^+$ then $2\cO = \frM_1\frM_2$ with $\cO/\frM_1 \simeq \cO/\frM_2 \simeq \F_4$ and $\frM_1 = \overline{\frM}_1$ and $\frM_2 = \overline{\frM}_2$;
        \item if $2$ is ramified in $K^+$ then $2\cO = \frM^2$ with $\cO/\frM \simeq \F_4$ and $\frM = \overline{\frM}$;
        \end{itemize}
        \item\label{2inPB} If $f(t) \in\PirrB$ then there is a unique maximal ideal $\frm$ of $\cO^+$ above $2$ which can either ramify, split or stay inert in $K$.
    \end{enumerate}
\end{proposition}

\begin{proof}
    By Lemma~\ref{lemma:splitting_mod_2}, $2$ is not inert in $K$.
    We now show that $2$ cannot be totally split in $K$.
    If $f(t)\in \PirrB$ then $2$ ramifies in $K^+$ by Lemma~\ref{lem:ram_ram} and we are done.
    Assume then that $f(t)\in \PirrA$.
    Then, by Lemma~\ref{lemma:MinK}, $K$ contains $M=\Q(\zeta_3)$.
    Note that $2$ is inert in $M$, so $2$ is not totally split in~$K$.

    Since $K$ is Galois by Proposition~\ref{prop:galois}, then the possible factorisations are $2\cO = \frM^4$, $2\cO = \frM_1^2\frM_2^2$, $2\cO = \frM_1\frM_2$ and $2\cO = \frM^2$.

    We start by proving \ref{2inPA}.   
    So, assume that $f(t) \in \PirrA$.
    Then $M=\Q(\zeta_3)$ is a subfield of $K$ by Lemma~\ref{lemma:MinK}.
    Since $2$ is inert in $M$ then $2$ cannot be totally ramified in $K$.
    Also, $K/K^+$ is unramified by Lemma~\ref{lem:if_ram_then_B}. 
    If $2$ is inert in $K^+$ then we must have $2\cO = \frM_1\frM_2$, with $\frM_1 = \bar\frM_2$.
    If $2$ splits in $K^+$ then $2\cO= \frm_1\frm_2$, that is each maximal ideal of $O^+$ above $2$ stays inert in $K$.
    In particular, $\frM_1=\overline\frM_1$ and $\frM_2=\overline\frM_2$.
    If $2$ is ramified in $K^+$ then $2\cO = \frM^2$ must occur and $\frM$ is stable under conjugation.
    Part \ref{2inPB} is Lemma~\ref{lem:ram_ram}.
    A search on the Weil polynomials shows that all listed factorisation occur; see Example~\ref{ex:table}. 
\end{proof}

\begin{figure}[h]
    \begin{tikzpicture}
        \node (K) at (-5.5,3) {$K$};
        \node (K+) at (-5.5,0) {$K^+$};
        \node (Q) at (-5.5,-3) {$\Q$};
        \draw[-] (K) to (K+);
        \draw[-] (K+) to (Q);

        \node (two) at (0,-3) {$2$};

        \node (split) at (0,0) {$\frm_1\frm_2$};
        \node (splitinert) at (0,3) {$\frm_1\frm_2$};
        \draw[-] (two) -- (split) node[midway,rectangle,draw=white, fill=white] {\textbf{npp}};
        \draw[-] (split) to (splitinert) ;

        \node (inert) at (-3,0) {$\frm=2\cO^+$};
        \node (inertsplit) at (-3,3) {$\frM_1\frM_2$};
        \draw[-] (two) -- (inert) node[midway,left] {\textbf{npp}};
        \draw[-] (inert) to (inertsplit);

        \node (ram) at (3,0) {$\frm^2$};
        \node (raminert) at (1.6,3) {$\frm^2$};
        \node (ramsplit) at (3,3) {$\frM_1^2\frM_2^2$};
        \node (ramram) at (4.4,3) {$\frM^4$};
        \draw[-] (two) -- (ram) node[midway,right] {\small \textbf{npp,Wres}};
        \draw[-] (ram) -- (ramsplit) node[midway,rectangle,draw=white, fill=white] {\small \textbf{Wres}};
        \draw[-] (ram) -- (ramram) node[midway,right] {\small \textbf{Wres}};
        \draw[-] (ram) -- (raminert) node[midway,left] {\small \textbf{npp,Wres}};
    \end{tikzpicture}
    \caption{Possible factorisations of the prime $2$ in $K/K^+/\Q$.}
\end{figure}
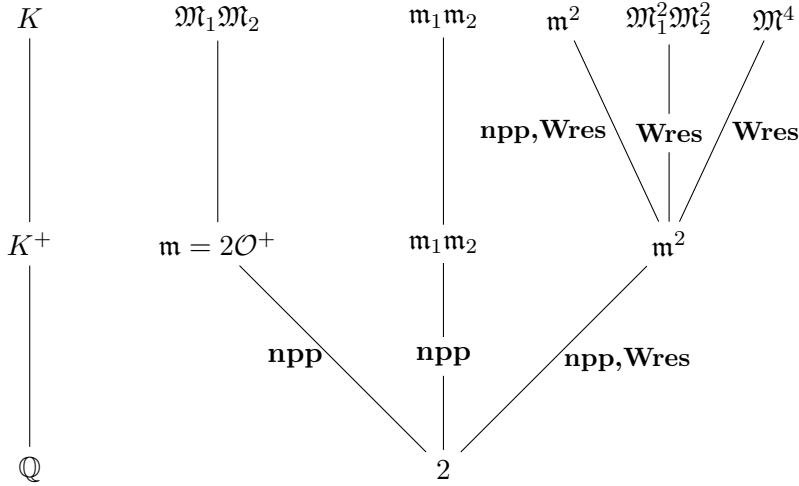

\section{Kernels of polarisations of abelian varieties}\label{sec:kernels}
In \cite{Howe95, Howe96}, Howe studied when an isogeny class of abelian varieties contains at least one variety with a polarisation with prescribed kernel by means of Grothendieck groups of finite modules over orders associated to the isogeny class.
In this section, everything is recalled from \cite{Howe96}, but Theorem~\ref{thm:deg_ord}, Lemmas~\ref{lemma:tensor_ZR} and \ref{lem:P_is_tensor_lin_fact} which are novel, to the best of our knowledge.

\smallskip
Fix an irreducible Weil polynomial $f(t)$ determining an isogeny class $\cC$ of abelian varieties over $\Fq$.
We stress that in this section we do not make assumptions on the dimension.
We set, as usual, $K=\Q[x]/f(t)=\Q(\pi)$ and $K^+=\Q(\pi+q/\pi)$.
Let $R=\Z[\pi,q/\pi]\subset K$ and $R^+=\Z[\pi+q/\pi]\subset K^+$.
In what follows, $S$ will denote a generic order in $K$, possibly satisfying additional hypothesis.
Denote by $\Mod_S$ the category of finite length $S$-modules.
Since $S$ is an order, it follows that an $S$-module $M$ has finite length if and only if $M$ is a finite set.
We denote by $|M|$ the number of elements of $M$.

Let $G(\Mod_S)$ be the Grothendieck group of $\Mod_S$ which is defined as the quotient of the free abelian group on the isomorphism classes of objects in $\Mod_S$ by the subgroup generated by the expressions $M-M'-M''$ for all exact sequences $0\to M'\to M\to M''\to 0$ in $\Mod_S$.
For $M\in\Mod_S$, we denote its class in $G(\Mod_S)$ by $[M]_S$.
The association $M\mapsto |M|$ induces a group homomorphism $G(\Mod_S) \to \Q^*$.
Note that $G(\Mod_S)$ is a free abelian group on the simple objects of $\Mod_S$.
Every such simple object is $S$-linearly isomorphic to $S/\frN$ for a maximal $S$-ideal~$\frN$.
An element of $G(\Mod_S)$ is called \emph{effective} if it is a sum of positive multiples of simple objects.
Let $\alpha=a/b$ be an element of $K^\times$.
The \emph{principal element} generated by $\alpha$ is the element $\Pr_S(\alpha)=[S/a S]_S-[S/bS]_R$ of $G(\Mod_S)$.
Note that $\Pr_S$ induces a group homomorphism from $K^\times$ to $G(\Mod_S)$.

Assume now that the order $S$ satisfies $S=\overline{S}$.
For $M \in \Mod_S$, define $\widehat{M}=\Hom_\Z(M,\Q/\Z)$ where the $S$-module structure is defined by $(r\psi)(m)=\psi(\overline{r}m)$ for every $\psi\in \widehat{M}$, $r\in S$ and $m \in M$.
The association $M\mapsto \widehat{M}$ induces a duality on $\Mod_S$, which defines an involution $\overline \cdot$ on $G(\Mod_S)$ by setting $\overline{[M]}_S=[\widehat{M}]_S$.
An element $P$ of $G(\Mod_S)$ is \emph{symmetric} if $P=\overline P$.

\begin{lemma}\label{lemma:symmetric}
    Let $S$ be an order in $K$ such that $S=\overline S$.
    Let $\frM$ a maximal ideal of $S$ and $I \subseteq S$ be a fractional $S$-ideal.
    Then
    \[ \overline{[S/\frM]}_S=[S/\overline{\frM}]_S \] 
    and
    \[ \overline{[S/I]}_S = [S/\overline I]_S.  \] 
\end{lemma}
\begin{proof}
    Observe that $\overline{[S/\frM]}_S=[S/\overline{\frM}]_S$ if and only if there is an $S$-linear isomorphism
    \[ \widehat{S/\frM} \simeq S/\overline{\frM}. \]
    Let $\psi:S/\frM \to \Q/\Z$ be a non-zero additive homomorphism.
    Consider the $S$-linear morphism $f:S\to \widehat{S/\frM}$ defined by $f(1) = \psi$.
    Note that $\overline\frM \subseteq \ker(f)$.
    Since $\frM$ is maximal and $f$ is non-zero, we get that $\overline{\frM} = \ker(f)$.
    By Pontrjagin duality, 
    we have that 
    \[ |\widehat{S/\frM}| = |S/\overline\frM|. \]
    Hence the injective $S$-linear morphism $S/\overline\frM \into \widehat{S/\frM}$ induced by $f$ is an isomorphism.  

    For the second equality in the statement, write
    \[ [S/I]_S = \sum_{i=1}^n n_i[ S/\frM_i]_S, \]
    where the $\frM_i$'s are the maximal ideals of $S$ containing $I$, and $n_i$ is the length of $(S/I)_{\frM_i}$ as an $S_{\frM_i}$-module.
    Then 
    \[ \overline{[S/I]}_S = \sum_{i=1}^n n_i\overline{[S/\frM_i]}_R= \sum_{i=1}^n n_i[S/\overline{\frM}_i]_S= [S/\overline{I}]_S,\]
    where the central equality holds by the first part.
\end{proof}

Again, let $S$ be an order in $K$ such that $S=\overline S$.
Set $S^+=S \cap K^+$. 
By considering finite the $S$-modules as finite $S^+$-modules, we obtain the \emph{norm} homomorphism $N_{S/S^+}([M]_S)=[M]_{S^+}$ from $G(\Mod_S)$ to $G(\Mod_{S^+})$.
Define $Z(S)$ as the set of symmetric elements of 
\[ \ker\left( G(\Mod_S) \overset{N_{S/S^+}\otimes \frac{\Z}{2\Z}}{\longrightarrow} G(\Mod_{S^+})\otimes_\Z \frac{\Z}{2\Z} \right). \]
Let $B(S)$ be the subgroup $\{ P+\overline P : P \in G(\Mod_S) \}$ of $Z(S)$.
Finally, set
\[ \cB(S)=\frac{Z(S)}{(B(S)\cdot \Pr_S(K^\dagger))}, \]
where $K^\dagger$ is the groups of squares of totally positive elements of $K^+$.

Let $\KerC$ be the category whose objects are finite commutative group schemes that can be embedded as closed subgroup-schemes in some abelian variety in the isogeny class $\cC$.
Every object of $\KerC$ is of the form $\ker(\varphi)$ for some isogeny $\varphi:A\to B$, where $A$ and $B$ are elements in $\cC$.

\begin{definition}
    A finite group scheme $G$ is \emph{attainable} in $\cC$ if there is a variety in $\cC$ with a polarisation whose kernel is isomorphic to $G$.
\end{definition}

Let $G(\KerC)$ be the associated Grothendieck group, which is defined analogously to the one of an order.
For an element $G$ in $\KerC$ we denote its class in $G(\KerC)$ by $[G]_\cC$.
Moreover, we denote by $\widehat G$ its Cartier dual.

The category $\KerC$ splits into a product of four subcategories $\mathcal{K}_{rr}$, $\mathcal{K}_{rl}$, $\mathcal{K}_{lr}$ and $\mathcal{K}_{ll}$, whose objects are respectively reduced elements of $\KerC$ with reduced Cartier dual, reduced elements of $\KerC$ with local Cartier dual, local elements of $\KerC$ with reduced Cartier dual, and local elements of $\KerC$ with local Cartier dual.
For an isogeny $\varphi:A\to B$, the set of geometric points $G(\overline \Fq)$ of $G=\ker\varphi$ admits a natural $R$-module structure by identifying $\pi$ with the Frobenius of $A$.
If $\mathcal{K}_{ll}$ is non-empty then it contains a unique simple element $\alpha_p$.
This occurs if and only if $\cC$ is non-ordinary.

A crucial point of \cite{Howe96} is that $G(\KerC)$ and $G(\Mod_R)$ are isomorphic, as we now recall.
Consider the association given by the following rules:
\begin{itemize}
    \item for $G$ in $\mathcal{K}_{rr}\times \mathcal{K}_{rl}$, set $\epsilon([G]_\cC) = [ G(\overline \Fq)  ]_R $,
    \item for $G$ in $\mathcal{K}_{lr}$, set $\epsilon([G]_\cC) = [ \widehat{\widehat{G}(\overline \Fq)}  ]_R $,
    \item set $\epsilon([\alpha_p]_\cC) = [ M ]_R $ where $M=\Z/p\Z$ with the $R$-module structure induced by letting $\pi$ and $q/\pi$ act as $0$.
\end{itemize}
By extending $\epsilon$ by linearity to $G(\KerC)$, we obtain a group homomorphism $\epsilon: G(\KerC)\to G(\Mod_R)$.

\begin{theorem}[{\cite[Thm.~3.5]{Howe96}}]\label{thm:howe_isom}
    The group homomorphism $\epsilon : G(\KerC) \to G(\Mod_R)$ is an isomorphism.
\end{theorem}

\begin{definition}
    A finite $R$-module $M$ is \emph{attainable} if $\epsilon([G]_\cC) = [M]_R$ for some attainable $G$ in $\cC$.
\end{definition}

By the very construction of $\epsilon$, we can deduce the following result.
\begin{theorem}\label{thm:deg_ord}
    Let $\varphi:A\to B$ be an isogeny between elements of $\cC$.
    Let $M$ be a finite $R$-module such that $\epsilon([\ker \varphi]_\cC) = [M]_R$.
    Then
    \[ \deg(\varphi) = |M|. \]
\end{theorem}
\begin{proof}
    Set $G = \ker\varphi$, so that $\deg(\varphi)$ equals the rank $\rk(G)$ of $G$ as a group scheme.
    There exist a non-negative integer $n$ and a group scheme $G' \in \mathcal{K}_{rr}\times \mathcal{K}_{rl} \times \mathcal{K}_{lr}$ such that
    \[ [G]_\cC = [ \alpha_p^n \times G' ]_\cC = n[ \alpha_p ]_\cC \oplus [ G' ]_\cC. \]
    It follows that
    \begin{equation}\label{eq:rk}
        \rk(G) =\rk(\alpha_p^n)\cdot \rk(G') = p^n\rk(G').
    \end{equation}
    Let $M_p$ be $\Z/p \Z$ with the trivial $R$-module structure and $M'$ be a finite $R$-module such that $\epsilon([G']_\cC) = [M']_R$.
    By Theorem~\ref{thm:howe_isom}, we have
    \[ [M]_R = n[M_p]_R \oplus [M']_R. \]
    Hence
    \begin{equation}\label{eq:ord}
        |M| = |M_p|^n \cdot |M'| = p^n|M'|.
    \end{equation}
    Since, by construction of $M'$, we have that $\rk(G')=|M'|$, by combining Equations~\eqref{eq:rk} and~\eqref{eq:ord}  it follows that $\rk(G)=|M|$.
\end{proof}

Since $R$ is a locally free $R^+$-module, the association $M \mapsto M \otimes_{R^+} R$ induces a well defined homomorphism $t_{R/R^+}:G(\Mod_{R^+})\to G(\Mod_R)$.

\begin{lemma}\label{lemma:tensor_ZR}
    Let $S$ be an order in $K$ such that $S=\overline S$ and that $S$ is locally free of rank $2$ over $S^+=S \cap K^+$.
    Then $Z(S) = t_{S/S^+}(G(\Mod_{S^+}))$.
    Moreover, if $M \in \Mod_S$ is such that $[M]_S \in Z(S)$ then $[M]_S=t_{S/S^+}(P')$ for some effective $P' \in G(\Mod_{S^+})$.
\end{lemma}
\begin{proof}
    Let $P'$ be an element of $G(\Mod_{S^+})$.
    Write 
    \[ P' = \sum_\frn n_{\frn} [S^+/\frn]_{S^+}, \]
    where $\frn$ runs over the maximal ideals of $S^+$.
    Then 
    \[ t_{S/S^+}(P') = \sum_\frn n_{\frn} [S^+/\frn \otimes_{S^+} S]_{S}. \]
    We have a natural $S$-linear isomorphism $S^+/\frn \otimes_{S^+} S \simeq S/\frn S$.
    So
    \[ N_{S/S^+}([S/\frn S]_{S}) = 2[S^+/\frn]_{S^+}\]
    because $S$ is locally free of rank $2$ over $S^+$.
    Hence $N_{S/S^+}([S/\frn S]_{S}) \in 2G(\Mod_{S^+})$.
    Since $S=\overline{S}$ and $\frn S=\overline{\frn S}$, by Lemma~\ref{lemma:symmetric}, we get that $[S^+/\frn \otimes_{S^+} S]_S = [S/\frn S]_S$ is symmetric.
    This concludes the proof that
    \[ t_{S/S^+}(G(\Mod_{S^+})) \subseteq Z(S).\]

    Now we prove the reverse inclusion.
    Let $P$ be a symmetric element of $G(\Mod_S)$.
    Write
    \[ P = \sum_\frN n_\frN [S/\frN]_S, \]
    where the sum is taken over all maximal ideals $\frN$ of $S$.
    Since $P$ is symmetric we must have $n_\frN = n_{\overline{\frN}}$ for every maximal ideal $\frN$ of $S$.
    Write
    \[ P = \sum_{\frN\neq \overline\frN} n_\frN \left( [S/\frN \oplus S/\overline{\frN}]_S\right) + \sum_{\frN = \overline\frN} n_\frN \left( [S/\frN]_S\right). \]
    Fix a maximal ideal $\frN$ of $S$ and set $\frn = \frN \cap S^+ = \overline{\frN} \cap S^+$.
    
    Assume first that $\frN \neq \overline{\frN}$.
    Then $\frn S \subseteq \frN \overline{\frN}$.
    Hence we have a surjective map 
    \[ S /\frn S \onto S/\frN\overline{\frN}\simeq S/\frN \times S/\overline\frN. \]
    Since $S/\frn S$ is a $2$-dimensional vector space over $S^+/\frn$, we get that $S/\frN\simeq S/\overline\frN\simeq S^+/\frn$.
    This means that $\frn S=\frN\overline{\frN}$.
    Therefore
    \[ [S/\frN \oplus S/\overline{\frN}]_S = [S/\frN\overline{\frN}]_S = [S/\frn S]_S = t_{S/S^+}([S^+/\frn]_{S^+}). \]
    
    Now consider the case $\frN = \overline{\frN}$.
    Observe that the residue field extension $S^+/\frn \to S/\frN$ is of degree $1$ or $2$.
    Let $\mathcal{M}_1$ (resp.~$\mathcal{M}_2$) be the set of maximal ideals $\frN$ of $S$ such that $\frN = \overline{\frN}$ and the degree is $1$ (resp.~$2$).
    If $\frN \in \mathcal{M}_2$ then $\frN = \frn S$, since $S/\frn S$ is has dimension $2$ over $S^+/\frn$.
    Hence
    \[ [S/\frN]_S = [S/\frn S]_S = t_{S/S^+}([S^+/\frn]_{S^+}).\]
    By the first part of the proof, we get that $P=\overline{P}\in Z(S)$ if and only if
    \[ \sum_{\frN \in \mathcal{M}_1} n_\frN \left( [S/\frN]_S\right) =
    P - \sum_{\frN\neq \overline\frN} n_\frN \left( [S/\frN \oplus S/\overline{\frN}]_S\right) - \sum_{\frN \in \mathcal{M}_2} n_\frN \left( [S/\frN]_S\right) \]
    is in $Z(S)$ as well.
    This happens precisely when $n_\frN N_{S/S^+}\left( [S/\frN]_S\right)$ is in $2G(\Mod_{S^+})$ for every $\frN \in \mathcal{M}_1$, that is, when each $n_\frN$ is even, since $S/\frN \simeq S^+/\frn$ is a simple $S^+$-module.
    Since, for $\frN \in \mathcal{M}_1$, we have $S$-linear isomorphisms
    \[ t_{S/S^+}([S^+/\frn]_{S^+}) \simeq S/\frn S \simeq (S/\frN)^2, \]
    we conclude the proof that $Z(S) \subseteq t_{S/S^+}(G(\Mod_{S^+}))$ by additivity of $t_{S/S^+}$.
\end{proof}

Recall that the order $R=\Z[\pi,q/\pi]$ is locally free of rank $2$ over $R^+=\Z[\pi +q/\pi]$.
\begin{lemma}\label{lem:P_is_tensor_lin_fact}
    Let $\ell$ be a rational prime.
    The following statements are equivalent.
    \begin{enumerate}[(i)]
        \item \label{lem:P_is_tensor_lin_fact:tensor} There exists a finite $R$-module $M$ of order $|M|=\ell^2$ such that the class of $M$ in $G(\Mod_{R})$ is of the form $t_{R/R^+}(P')$ for some effective $P'$ in $G(\Mod_{R^+})$.
        \item \label{lem:P_is_tensor_lin_fact:Z} There exists a finite $R$-module $M$ of order $|M|=\ell^2$ such that $[M]_R$ is in $Z(R)$. 
        \item \label{lem:P_is_tensor_lin_fact:lin} $f^+(t) \bmod \ell$ has a linear factor. 
    \end{enumerate}
    If any of the equivalent statements hold then $[M]_R = t_{R/R^+}([R^+/\frl]_{R^+})$ where $\frl$ is the unique maximal ideal of $R^+$ above $\ell$ corresponding to any linear factor of $f^+(t) \bmod \ell$.
\end{lemma}
\begin{proof}
    The equivalence of \ref{lem:P_is_tensor_lin_fact:tensor} and \ref{lem:P_is_tensor_lin_fact:Z} is an immediate consequence of Lemma~\ref{lemma:tensor_ZR}.

    Let $P' \in G(\Mod_{R^+})$ be effective and non-zero.
    Then there are positive integers $n_1,\ldots,n_r$ and maximal $R^+$-ideals $\frn_1,\ldots,\frn_r$ such that 
    \[ P' = \sum_{i=1}^r n_i[R^+/\frn_i]. \]
    Hence
    \[ t_{R/R^+}(P') = \sum_{i=1}^r n_i[R^+/\frn_i \otimes_{R^+} R]. \]
    Since $R$ is locally free of rank $2$ over $R^+$, we get that
    \[ |R^+/\frn_i \otimes_{R^+} R| = |R^+/\frn_i|^2. \]
    If $M$ is a finite $R$-module whose image in $G(\Mod_{R})$ is $t_{R/R^+}(P')$ for $P'$ as above then
    \[ |M| = \prod_{i=1}^r |R^+/\frn_i|^{2n_i}. \]

    We now show that \ref{lem:P_is_tensor_lin_fact:tensor} implies \ref{lem:P_is_tensor_lin_fact:lin}.
    So, assume also that $M$ as above satisfies $|M|=\ell^2$ for a rational prime $\ell$.
    Then $r=1$, $n_1=1$ and $R^+/\frn_1 \simeq \F_\ell$.  
    This implies that $f^+(t) \bmod \ell$ has a linear factor (cf.~\cite[Thm.~8.2]{Stevenhagen08}).

    Finally, assume that \ref{lem:P_is_tensor_lin_fact:lin} holds.
    Let $\frl$ be the corresponding maximal ideal of $R^+$.
    Then we see that $R^+/\frl\otimes_{R^+}R$ is an $R$-module of order $\ell^2$, as in \ref{lem:P_is_tensor_lin_fact:tensor}.
\end{proof}

The following result completely describes $\cB(\cO)$ for the maximal order $\cO$ of $K$.
\begin{proposition}\label{prop:BO}
    If $K/K^+$ is ramified at a finite prime then $\cB(\cO)=0$.
    Otherwise, the Artin map induces an isomorphism $\cB(\cO)\simeq \Gal(K/K^+)$.
\end{proposition}
\begin{proof}
    This is part of \cite[Prop.~6.2]{Howe96}.
\end{proof}
Define 
\[ H(R) = \frac{Z(R)}{B(R)}. \]
The group $H(R)$ is a vector space over $\F_2$ whose basis is given by the classes of the simple $R$-modules $R/\frn$ where $\frn$ is a maximal $R$-ideal such that $\frn = \overline \frn$ and the index $[R/\frn : R^+/(R^+\cap \frn)]$ of residue fields is even.
Such maximal ideals are called the \emph{generating primes} of $H(R)$.

The inclusion $i\colon R\to \cO$ induces a group homomorphism $i^*\colon\cB(\cO)\to\cB(R)$ by considering every finite $\cO$-module as an $R$-module.

Let $\psi\colon H(R) \to \cB(R)$ and $\chi\colon H(\cO) \to \cB(\cO)$ denote the canonical reductions.
Note that $i$ induces a norm map $G(\Mod_{\cO}) \to G(\Mod_{R})$ which in turn induces a norm $N\colon H(\cO) \to H(R)$.
\begin{proposition}\label{prop:ses_BR}
    We have an exact sequence
    \[ H(\cO) \overset{(N,-\chi)}{\longrightarrow} H(R) \oplus \cB(\cO) \overset{\psi \oplus i^*}{\longrightarrow} \cB(R) \to 0. \]
\end{proposition}
\begin{proof}
    This is a special case of \cite[Prop.~6.4]{Howe96}.
\end{proof}

We have the following:
\begin{theorem}[{\cite[Thm.~1.3]{Howe96}}]\label{thm:attainable}
    There is an element $I_\cC$ of $\cB(R)$ such that the elements of $G(\Mod_R)$ that are attainable in $\cC$ are precisely the effective elements of $Z(R)$ that map to $I_\cC$ in $\cB(R)$.
    In particular, the isogeny class $\cC$ contains a principal polarisation if and only if $I_\cC=0$. 
\end{theorem}

\begin{proposition}\label{prop:IC_in_BO}
    The obstruction element $I_\cC$ lies in $i^*(\cB(\cO))$.
\end{proposition}
\begin{proof}
    This is part of {\cite[Prop.~7.1]{Howe96}}.
\end{proof}

\section{Polarisations of degree \texorpdfstring{$4$}{4} on abelian surfaces}\label{sec:pols4}
In Section \ref{sec:g3onsurf}, we showed that an abelian surface with Weil polynomial $f\in\PirrA\sqcup\PirrB$ contains an $\Fq$-irreducible curve of arithmetic genus $3$ if and only if it admits a polarisation with kernel of order $4$. 
We now apply the results developed in Section~\ref{sec:kernels} to characterise isogeny classes defined by a polynomial in $\PirrA\sqcup\PirrB$ not containing abelian surfaces with an $\Fq$-irreducible curve of genus $3$ lying on.

In this section, we assume that $f(t) \in \PirrA\sqcup\PirrB$.
As before, we denote the corresponding isogeny class by $\cC$ and set $K=\Q[x]/f(t)=\Q(\pi)$ and $K^+=\Q(\pi+q/\pi)$.
We fix also $R=\Z[\pi,q/\pi] \subset K$ and $R^+=\Z[\pi+q/\pi] \subset K$.

\begin{proposition}\label{prop:A_general}
    Let $f(t)\in\PirrA$.
    Then $2$ is inert in $K^+$ if and only if there is no attainable $R$-module of order $4$.
\end{proposition}
\begin{proof}
    If $f^+(t) \bmod 2$ is irreducible then $2$ is inert in $K^+$ and there is no attainable $R$-module of order $4$ by Lemma~\ref{lem:P_is_tensor_lin_fact} and Theorem~\ref{thm:attainable}.
    So, we assume that $f^+(t)\bmod 2$ has a linear factor for the rest of the proof.
    Hence, again by Lemma~\ref{lem:P_is_tensor_lin_fact} and Theorem~\ref{thm:attainable}, if there is an attainable $R$-module $M$ of order $4$ then we must have
    \[ [M]_R=[R^+/\frl\otimes_{R^+}R]_R=[R/\frl R]_R, \]
    where $\frl$ is a maximal ideal of $R^+$ such that $R^+/\frl\simeq \F_2$. 
    By Lemma~\ref{lem:if_ram_then_B}, $K/K^+$ is unramified.
    So, the Artin map induces an isomorphism $\Gal(K/K^+)\simeq \cB(\cO)$ by Proposition \ref{prop:BO}.
    Let $\alpha$ be the non-zero element of $\cB(\cO)$, which corresponds to the Artin symbol of any maximal ideal of $\cO^+$ that stays inert in $K$.
    By Theorem~\ref{thm:attainable} and Proposition~\ref{prop:IC_in_BO}, $I_\cC$ is non-zero and in $i^*(\cB(\cO))$.
    Hence $I_\cC = i^*(\alpha)$.
    Let $z_\frl$ be the image of $[R/\frl R]_R$ in $\cB(R)$.
    Note that 
    \[ z_\frl = (\Psi \oplus i^*)((x_\frl,0)), \]
    where $x_\frl$ is the image of $[R/\frl R]_R$ in $H(R)$.
    We obtain that $z_\frl - I_\cC = (\Psi \oplus i^*)((x_\frl,\alpha))$.
    Hence there is an $R$-module $M$ of order $4$ which is attainable if and only if 
    \[ (x_\frl,\alpha) \in \ker((\Psi \oplus i^*)) = (N,-\chi)(H(\cO)), \] 
    where the equality holds by Proposition~\ref{prop:ses_BR}.

    Assume that $2$ is inert in $K^+$ and write $\frm=2\cO^+$.
    By Proposition~\ref{prop:fact_2}, we see that $\frm$ splits in $K$, say $\frm\cO = \frM\overline{\frM}$.
    A preimage of $x_\frl$ in $H(\cO)$ via $N:H(\cO)\to H(R)$ must be an $\F_2$-linear combination of the images $y_\frM$ and $y_{\overline{\frM}}$ of $\frM$ and $\overline{\frM}$ in $H(\cO)$.
    Since $\frm$ is split in $K$, then the Artin symbols of~$\frM$ and $\overline{\frM}$ are trivial in $\Gal(K/K^+)\simeq\cB(\cO)$.
    Hence any $\F_2$-linear combination of the images $y_\frM$ and $y_{\overline{\frM}}$ won't be mapped by $\chi:H(\cO)\to \cB(\cO)$ to $\alpha$.
    This means that $(x_\frl,\alpha)$ is not in $(N,-\chi)(H(\cO))$.
    Hence, there is no $R$-module $M$ of order $4$ which is attainable.

    Assume now that $2$ is not inert in $K^+$.
    Let $\frM$ be a maximal ideal of $\cO$ containing $\frl R$.
    Denote the image of $[\cO/\frM]_\cO$ in $H(\cO)$ by $y_\frM$.
    By Proposition~\ref{prop:fact_2}, we have that $\frM=\frm\cO$ for a maximal ideal $\frm$ of $\cO^+$.
    This implies that $\chi(y_\frM)=\alpha$.
    So, to prove that $R/\frl R$ is attainable, we are left to show that $N(y_\frM)=x_\frl$.
    We first consider the case that $\frl R$ is a maximal ideal of $R$.
    Note that $\frl R=\frM\cap R$ and we have an isomorphism $R/\frl R\simeq \cO/\frM$.
    Hence $N(y_\frM)=x_\frl$.
    Therefore, $R/\frl R$ is attainable.
    Now we consider the case when $\frl R$ is not a maximal ideal of $R$.
    Set $\frL=\frM\cap R$.
    The natural surjective map $R/\frl R \onto R/\frL$ is not an isomorphism.
    Since $\frL=\overline\frL$ and $[R/\frl R]_R\in Z(R)$ we get that
    \[ [R/\frl R]_R = 2[R/\frL]_R. \]
    Moreover, also $[\cO/\frM]_R = 2[R/\frL]_R$.
    Hence, again, $N(y_\frM)=x_\frl$.
    Therefore, as before, $R/\frl R$ is attainable.
\end{proof}

\begin{proposition}\label{prop:B_general}
    Assume that $f(t)\in\PirrB$.
    Write $2\cO^+=\frm^2$ and let $\frl=\frm\cap R^+$ be the unique maximal ideal of $R^+$ above $2$.
    Then there is an $R$-module $M$ of order $4$ which is attainable
    if and only if $\frl R$ is not a maximal ideal of $R$.
\end{proposition}
\begin{proof}
    Observe that $f(t) \bmod 2$ is a square of a linear factor.
    Let $\frl$ be the corresponding maximal ideal of $R^+$.
    By Lemma~\ref{lem:P_is_tensor_lin_fact} and Theorem~\ref{thm:attainable}, if there is an attainable $R$-module $M$ of order $4$ then we must have
    \[ [M]_R=[R^+/\frl\otimes_{R^+}R]_R=[R/\frl R]_R. \]
    Since $f(t)\in\PirrB$, the obstruction element $I_\cC$ is trivial by Theorem~\ref{thm:attainable}.
    Hence, there is an $R$-module $M$ of order $4$ which is attainable if and only if the image $z_\frl$ of $[R/\frl R]_R$ in $\cB(R)$ is trivial.
    Note that
    \[ z_\frl = (\Psi \oplus i^*)((x_\frl,0)), \]
    where $x_\frl$ is the image of $[R/\frl R]_R$ in $H(R)$.
    Also, the $R$-module $R/\frl R$ has length at most $2$ and that it has length $1$ if and only if $\frl R$ is a maximal ideal of $R$.

    If the length of $R/\frl R$ is $2$ then $[R/\frl R]_R=[R/\frn]_R+[R/\overline{\frn}]_R$ for a maximal ideal $\frn$ of $R$, with possibly $\frn=\overline{\frn}$, and so, by Lemma~\ref{lemma:symmetric}, $[R/\frl R]_R$ is in $B(R)$.
    Hence, $x_\frl$ is trivial in $H(R)$ and so $z_\frl = I_\cC$ in $\cB(R)$.
    This means that $R/\frl R$ is an attainable $R$-module of order~$4$.

    To conclude the proof, we need to show that if the length of $R/\frl R$ is $1$, that is, $\frl R$ is a maximal $R$-ideal, then $(x_\frl,0)$ is not in 
    \[ \ker((\Psi \oplus i^*)) = (N,-\chi)(H(\cO)), \] 
    where the equality holds by Proposition~\ref{prop:ses_BR}.
    So, assume that $\frl R$ is a maximal ideal of $R$.
    Let $\frM $ be a maximal ideal of $\cO$ above $\frl R$.
    Since we have an inclusion of residue field $R/\frl R\to \cO/\frM$ and $2$ is not inert in $\cO$ 
    by Proposition~\ref{prop:fact_2}, then 
    \[ R/\frl R \simeq \cO/\frM \simeq \F_4. \]
    This implies that $[\cO/\frM]_R = [R/\frl R]_R$.
    If we denote the image of $[\cO/\frM]_{\cO}$ in $H(\cO)$ by $y_\frM$, we get $N(y_\frM) = x_\frl$. 
    By Proposition~\ref{prop:fact_2}, since $\cO/\frM\simeq\F_4$, we get that $\frm$ is inert in $K$.
    This means that the Artin symbol $\left(\frac{K/K^+}{\frm}\right)$ is a generator of $\Gal(K/K^+)$.
    We now claim that $K/K^+$ is unramified.
    If not, then by Lemma~\ref{lem:ram_ram} and Lemma~\ref{lem:ram_tot_ram}, we would get that there is unique maximal ideal $\frm$ of $\cO^+$ above $2$ which ramifies in $\cO$, say $\frm\cO = \frM^2$. 
    For such $\frM$ we would have $\cO/\frM\simeq \F_2$.
    Since we know that every maximal ideal of $\cO$ above $\frl$ has residue field isomorphic to $\F_4$, we deduce that $K/K^+$ is unramified.
    Then, by Proposition \ref{prop:BO}, we get that $\chi(y_\frM)\neq 0$ in $\cB(\cO)$.
    So $(x_\frl,0)$ is not in the image of $(N,-\chi)$ and, therefore, there is no $R$-module of order $4$ which is attainable.
\end{proof}

\smallskip

In Theorem~\ref{thm:kers_size_4}, we wrap-up the discussion above.
Moreover, in certain cases with $\cC$ ordinary, we show that if there is an abelian variety with polarisation with kernel of order $4$ then there is an abelian variety with polarisation with kernel of order $4$ and maximal endomorphism ring.
To do so, we use results from \cite{Howe95}. 
The construction of $\cB(\cO)$ in \cite{Howe95} is slightly different form the one in \cite{Howe96} that we have presented above, but certainly equivalent in view of Proposition~\ref{prop:BO} and \cite[Prop.~10.1]{Howe95}.

\begin{definition}
    A finite $R$-module $M$ is \emph{strongly attainable} 
    if there exists an abelian variety $B$ in $\cC$ such that $\End(B)$ is the maximal order of $K$ with a polarisation $\varphi:B\to B^\vee$ such that $\epsilon([\ker\varphi]_\cC) = [M]_R$.
\end{definition}

Clearly strongly attainable implies attainable, while the contrary does not hold in general.
A characterisation of strongly attainable $R$-modules is given in \cite[Prop.~5.7]{Howe95}, repeated below as Proposition~\ref{prop:strong_att} for convenience.
Such characterisation relates modules with a particular element $I_{K,\Phi}$ of $\cB(\cO)$ which depends on a choice of a CM-type $\Phi$ of $K$.
The definition of $I_{K,\Phi}$ is given in \cite[Def.~5.2]{Howe95}.
For our purposes, it sufficient to know whether it is trivial or not in $\cB(\cO)$.
\begin{lemma}\label{lem:IK}
    If $f\in \PirrA$ is ordinary then $I_{K,\Phi} \neq 0$.
    If $f\in \PirrB$ is ordinary and either 
    \begin{itemize}
        \item $K/K^+$ is ramified, or
        \item $K/K^+$ is unramified and there is no maximal ideal of $\cO^+$ dividing $\pi-q/\pi$ which stays inert in $K/K^+$,
    \end{itemize}
    then $I_{K,\Phi} = 0$.
\end{lemma}
\begin{proof}
    Assume that $f \in \PirrA$ is ordinary.
    Then $\cC$ does not contain a principally polarised abelian variety.
    The result follows from \cite[Prop.~11.3]{Howe95} and \cite[Cor.~11.4]{Howe95}.

    Assume from now that $f\in \PirrB$ is ordinary.
    If $K/K^+$ is ramified then the $I_{K,\Phi}=0$ by \cite[Prop.~11.1]{Howe95}.
    If $K/K^+$ is unramified and there is no maximal ideal of $\cO^+$ dividing $\pi-q/\pi$ which stays inert in $K/K^+$ then the result follows, again, by combining \cite[Prop.~11.3]{Howe95} and \cite[Cor.~11.4]{Howe95} and the observation that $\cC$ contains a principally polarised abelian variety.
\end{proof}

Recall that, by Proposition~\ref{prop:BO}, we have a natural surjective homomorphism from $G(\Mod_{\cO^+})$ to $\cB(\cO)$ defined by sending a maximal ideal to its Artin symbol.

\begin{proposition}{\cite[Prop.~5.7]{Howe95}}\label{prop:strong_att}
    Let $M$ be finite $R$-module.
    Then $M$ is strongly attainable if and only if $M\simeq \cO/\mathfrak{a}\cO$ for some $\cO^+$-ideal $\mathfrak{a} \subseteq \cO^+$ such that the image in $\cB(\cO)$ of the class of $\cO^+/\mathfrak{a}$ in $G(\Mod_{\cO^+})$ equals $I_{K,\Phi}$.
\end{proposition}

\begin{proposition}\label{prop:3B_ord}
    Assume that $f(t)\in\PirrB$ is ordinary.
    Write $2\cO^+=\frm^2$,
    Then either
    \begin{enumerate}[(a)]
        \item \label{prop:3B_ord:ram} $K/K^+$ is ramified and there is a strongly attainable $R$-module of order $4$, or
        \item \label{prop:3B_ord:unr_no_prime_stays_inert} $K/K^+$ is unramified, there is no maximal ideal of $\cO^+$ dividing $\pi-q/\pi$ which stays inert in $K$ and, there is a strongly attainable $R$-module of order $4$ if and only if $\frm$ splits in $K$.
    \end{enumerate}
\end{proposition}
\begin{proof}
    Note that $M=\cO/\frm\cO$ has order $4$.
    By Lemmas~\ref{lem:ram_tot_ram} and \ref{lem:if_ram_then_B}, either $K/K^+$ is ramified or, $K/K^+$ is unramified and there is no maximal ideal of $\cO^+$ dividing $\pi-q/\pi$ which stays inert in $K$.

    If $K/K^+$ is ramified then $\cB(\cO)=0$ by Proposition~\ref{prop:BO}.
    So $I_{K,\Phi}$ and the image in $\cB(\cO)$ of the class of $M$ in $G(\Mod_{\cO^+})$ are equal.
    Hence $M$ is a strongly attainable module of order $4$ by Proposition~\ref{prop:strong_att}.
    that is, we are in case~\ref{prop:3B_ord:ram}.
    
    Assume now that $K/K^+$ is unramified and there is no maximal ideal of $\cO^+$ dividing $\pi-q/\pi$ which stays inert in $K$.
    By Lemma~\ref{lem:IK}, we have that $I_{K,\Phi} = 0$.
    Since $K/K^+$ is unramified, the maximal ideal $\frm$ is either split or inert in $K$.
    Also, $\frm$ is split if and only if the Artin symbol of $\frm$ in $\Gal(K/K^+)$ is trivial. 
    By Proposition~\ref{prop:BO}, this happens precisely when the image in $\cB(\cO)$ of the class of $M$ in $G(\Mod_{\cO^+})$ is trivial, that is, equal to $I_{K,\Phi}$.
    We conclude the proof of \ref{prop:3B_ord:unr_no_prime_stays_inert} by applying Proposition~\ref{prop:strong_att}.
\end{proof}

The next theorem completes the proof of Main Theorem~\ref{mainthm2}.
\begin{theorem}\label{thm:kers_size_4}
    Assume that $f\in\PirrA\sqcup\PirrB$ and let $\cC$ be the corresponding isogeny class.
    \begin{enumerate}[(i)]
        \item \label{thm1:kers_size_4}            
            The following are equivalent:
            \begin{enumerate}[(a)]
                \item \label{thm1:kers_size_4:att} There is no attainable $R$-module of order $4$.
                \item \label{thm1:kers_size_4:pol} There is no $A$ in $\cC$ admitting a polarisation of degree $4$.
            \end{enumerate}
            Moreover, if $\cC$ is ordinary then \ref{thm1:kers_size_4:att} and \ref{thm1:kers_size_4:pol} are also equivalent to:
            \begin{enumerate}[(a)]
                \setcounter{enumii}{2}
                \item \label{thm1:kers_size_4:str_att} There is no strongly attainable $R$-module of order $4$.
                \item \label{thm1:kers_size_4:str_pol} There is no $A$ in $\cC$ with maximal $\Fq$-endomorphism ring admitting a polarisation of degree $4$.
            \end{enumerate}
        \item \label{thm1:kers_size_4:A}
            Assume that $f\in\PirrA$.
            Then \ref{thm1:kers_size_4:att} and \ref{thm1:kers_size_4:pol} are also equivalent to:
            \begin{enumerate}[(a)]
                \setcounter{enumii}{4}
                \item \label{thm1:kers_size_4:A:inert}  $2$ is inert in $K^+$.
            \end{enumerate}
        \item \label{thm1:kers_size_4:B}
            Assume that $f(t)\in \PirrB$ and write $f(t)=t^4+bt^2+q^2$. 
            Write $2\cO^+=\frm^2$ and let $\frl=\frm\cap R^+$ be the unique maximal ideal of $R^+$ above $2$. 
            Then \ref{thm1:kers_size_4:att} and \ref{thm1:kers_size_4:pol} are also equivalent to each of the following statements:
            \begin{enumerate}[(a)]
                \setcounter{enumii}{5}
                \item \label{thm1:kers_size_4:B:max} $\frl R$ is a maximal ideal of $R$.
                \item \label{thm1:kers_size_4:B:coeff} $b=1-2q$ and $q$ is odd, if $\cC$ is ordinary; $q$ is even, if $\cC$ is non-ordinary.
            \end{enumerate}
    \end{enumerate}
\end{theorem}
\begin{proof}
    The equivalences \ref{thm1:kers_size_4:att}$\iff$\ref{thm1:kers_size_4:pol} and \ref{thm1:kers_size_4:str_att}$\iff$\ref{thm1:kers_size_4:str_pol} 
    follow from the definitions of attainable and strongly attainable module and Theorem~\ref{thm:deg_ord}.
    The implications \ref{thm1:kers_size_4:att}$\Rightarrow$\ref{thm1:kers_size_4:str_att}  and \ref{thm1:kers_size_4:pol}$\Rightarrow$\ref{thm1:kers_size_4:str_pol} are clear.
    The reverse implications will be proven below distinguishing several cases.

    We now show Part \ref{thm1:kers_size_4:A}.
    So, assume that $f(t)\in\PirrA$.
    The equivalence \ref{thm1:kers_size_4:A:inert}$\iff$\ref{thm1:kers_size_4:att} is Proposition~\ref{prop:A_general}.
    We show that \ref{thm1:kers_size_4:str_att} implies \ref{thm1:kers_size_4:A:inert} by contraposition, when $\cC$ is ordinary.
    So, say that $2$ is not inert in $K^+$.
    Then, by Proposition~\ref{prop:fact_2}, there exists a maximal ideal~$\frl$ of $\cO^+$ above $2$ such that $\frl$ stays inert in $K$.
    Hence, the Artin symbol of $\frl$ is not trivial in $\Gal(K/K^+)$.
    Since $K/K^+$ is unramified by Lemma~\ref{lem:if_ram_then_B}, we get that the image in $\cB(\cO)$ of the class of $\cO^+/\frl$ in $G(\Mod_{\cO^+})$ is non trivial by Proposition~\ref{prop:BO}.
    Then by Lemma~\ref{lem:IK} and Proposition~\ref{prop:strong_att}, the module $\cO/\frl\cO$, which has order $4$, is strongly attainable.
    Therefore \ref{thm1:kers_size_4:str_att}$\Rightarrow$\ref{thm1:kers_size_4:att} and \ref{thm1:kers_size_4:str_pol}$\Rightarrow$\ref{thm1:kers_size_4:pol} for $f(t)\in\PirrA$.

    We now move to Part \ref{thm1:kers_size_4:B}.
    Assume now that $f\in \PirrB$.
    The equivalence \ref{thm1:kers_size_4:B:max}$\iff$\ref{thm1:kers_size_4:att} is Proposition~\ref{prop:B_general}.
    We deal with the ordinary and non-ordinary cases separately.
    
    Assume first that $\cC$ is ordinary.
    We now show \ref{thm1:kers_size_4:str_att}$\Rightarrow$\ref{thm1:kers_size_4:B:coeff}.
    So, we assume that there are no strongly attainable $R$-modules of order $4$.
    By Proposition \ref{prop:3B_ord}, this is equivalent to have $K/K^+$ unramified and $\frm$ inert in $K$.
    By Theorem~\ref{thm:no_g1g2} and Lemma~\ref{lem:Pirrordinary}, we get that $b=1-2q$.
    Hence $f(t) = t^4 +(1-2q)t^2 +q^2$.
    If $q$ is even, then $f(t)\equiv t^2(t+1)^2 \bmod 2$.
    The Kummer-Dedekind Theorem \cite[Thm.~8.2]{Stevenhagen08} implies that $\Z[\pi]$ has $2$ distinct maximal ideals above $2$.
    Hence the same holds for $R$.
    This cannot be the case since $\frL = R \cap \frm\cO$ is the unique maximal ideal of $R$ above $2$.
    Hence, $q$ is odd.
    This means that \ref{thm1:kers_size_4:B:coeff} holds.
    We now show that \ref{thm1:kers_size_4:B:coeff}$\Rightarrow$\ref{thm1:kers_size_4:B:max}.
    Then $f(t) \equiv (t^2+t+1)^2 \bmod 2$.
    The Kummer--Dedekind Theorem \cite[Thm.~8.2]{Stevenhagen08} implies that $\Z[\pi]$ has a unique maximal ideal above $2$ with residue field of order $4$.
    This implies that $\frL$ also has residue field of order $4$.
    Since $\frl R\subseteq \frL$ and $R/\frl R$ has also $4$ elements, we must have $\frl R =\frL$, that is, \ref{thm1:kers_size_4:B:max} holds.
    We deduce also that \ref{thm1:kers_size_4:str_att}$\Rightarrow$\ref{thm1:kers_size_4:att} and \ref{thm1:kers_size_4:str_pol}$\Rightarrow$\ref{thm1:kers_size_4:pol} for $f(t)\in\PirrB$ ordinary.

    Assume now that $\cC$ is non-ordinary.
    If $q$ is even then $f(t)\equiv t^4 \bmod 2$.
    This means that $R$ has a unique maximal ideal above $2$, which must be $\frL = (2,\pi,q/\pi) \subset R$, since $R/\frL \simeq \F_2$.
    Since $R/\frl R$ has $4$ elements, we obtain that $\frl R$ is not a maximal ideal of $R$.
    If $q$ is odd then $f(t)\equiv (t^2+t+1)^2$.
    By the Kummer--Dedekind Theorem \cite[Thm.~8.2]{Stevenhagen08}, the order $\Z[\pi]$ has a unique maximal ideal above $2$ which is regular and with residue field $\F_4$.
    It follows that the same hold for $R$.
    Hence $\frl R$ is maximal.
    This shows \ref{thm1:kers_size_4:B:coeff}$\iff$\ref{thm1:kers_size_4:B:max} and completes the proof of Part \ref{thm1:kers_size_4:B}.
\end{proof}
\begin{remark}\label{rmk:coeffs}
    In the case when $f(t)\in\PirrB$, we completely characterise when the equivalent conditions \ref{thm1:kers_size_4:att},\ref{thm1:kers_size_4:pol} and \ref{thm1:kers_size_4:B:max} hold in terms of the coefficients of $f(t)$ in \ref{thm1:kers_size_4:B:coeff}.
    It is easy to obtain a characterisation in terms of the coefficients of $f(t)=t^4+at^3+bt^2+aqt+q^2$ also when $f(t)\in\PirrA$.
    Indeed, if we write $\Delta_{f^+} = a^2 -4(b-2q) = c^2 d$ for integers $c$ and $d$ with $d$ squarefree, then it is well known that \ref{thm1:kers_size_4:A:inert} holds if and only if $d\equiv 5 \bmod 8$.
    Moreover, if $q$ is even (or equivalently $a$ is odd) then $f^+(t) \equiv t^2+t+1 \bmod 2$ is irreducible, which implies that $2$ is inert in $K^+$.
    Similarly, if $a$ is even and $a+b\not\equiv 1 \bmod 4$ then 
    $f^+(t) \equiv (t+1)^2$ and the remainder of the division of $f^+(t)$ by $t+1$ is not divisible by $4$.
    Hence, $2$ ramifies in $K^+$ by the Kummer--Dedekind Theorem \cite[Thm.~8.2]{Stevenhagen08}.
\end{remark}
\begin{remark}\label{rmk:compare_with_Howe}
   The statement Theorem~\ref{thm:kers_size_4}.\ref{thm1:kers_size_4:A} can be deduced from the claim contained in \cite[page 155, 2nd par.~section 3]{Howe95Bounds}, stated without a detailed proof.
\end{remark}

\section{Computing isomorphism classes admitting a polarisation of degree \texorpdfstring{$4$}{4}}\label{s:examples}
In \cite{Mar21}, it is described how to compute the isomorphism classes of abelian varieties over $\Fq$ belonging to an ordinary isogeny class $\cC$ determined by an irreducible Weil polynomial $f(t)$.
We summarise here the results that are relevant for us; see \cite[Cor.~4.4, Thm.~5.4]{Mar21}.

Let $K=\Q[t]/(f(t)) = \Q[\pi]$ where $\pi$ denotes the class of $t$ in $K$ and set $R=\Z[\pi,q/\pi]$.
There is an equivalence between the category of abelian varieties in $\cC$ (with $\Fq$-morphisms) and the category of fractional $R$-ideals in $K$ (with $R$-linear morphisms).
Hence, every overorder $S$ of $R$ occurs as the endomorphism ring of an abelian variety in $\cC$.
Moreover, the functor inducing the equivalence is compatible with duality and allows to describe polarisations as $R$-linear morphisms.
For example, if $\End(A)=S$ then $\End(A^\vee)=\overline{S}$.
In \cite[Sec.~6]{Mar21}, the equivalence is used to produce an algorithm to compute the abelian varieties in $\cC$ together with their polarisations (of a fixed degree) up to polarised isomorphism.
We will use this algorithm in the next examples.

\begin{remark}
    An analogous description of polarisations exists also for simple almost-ordinary abelian varieties over any finite field $\Fq$ of odd characteristic; see \cite{OswalShankar20}.
    This result does not apply to our case since none of the polynomials in $\PirrA\sqcup\PirrB$ is almost-ordinary by Lemma~\ref{lem:Pirrordinary}.
    A similar description also exists for abelian varieties over prime fields whose Weil polynomial does not have repeated complex roots, but only for polarisations of degree~$1$; see~\cite{BergKarMar21}. 
\end{remark}

\begin{example}\label{ex:controesempio}
    Consider the isogeny class of abelian surfaces over $\F_2$ with label \href{http://www.lmfdb.org/Variety/Abelian/Fq/2/2/a_ab}{2.2.a\_ab} determined by the Weil polynomial $f(t)=t^4 -t^2 + 4$.
    According to the LMFDB, the class contains a Jacobian.
    Moreover, one computes that the order $\Z[\pi,4/\pi]$ has $3$ overorders $S_1$, $S_2$ and $\cO$, where $\cO$ is the maximal order of the number field $K=\Q[t]/(f(t))$.
    One observe that $S_2 = \overline{S}_1$.
    Hence, each abelian surface $A$ with endomorphism ring $S_1$ cannot be isomorphic to its dual, which has endomorphism ring $S_2$.
    In particular, such an $A$ does not admit a principally polarisation. 
    Hence, such an $A$ does not contain an $\Fq$-irreducible curve of arithmetic genus $2$ by Proposition~\ref{prop:pp_iff_pa2} even if it is isogenous to a Jacobian; cf.~Remark~\ref{rmk:gap}.
\end{example}

\begin{example}
    In Theorem \ref{thm:kers_size_4}, we see that, in the isogeny class $\cC$ associated to an ordinary Weil polynomial in $\PirrA\sqcup\PirrB$, there is an abelian variety $A$ admitting a polarisation of degree $4$ if and only if there is an abelian variety $A'$ in $\cC$ with maximal endomorphism ring admitting a polarisation of degree $4$.
    This is not the case in general.
    For example, in the ordinary isogeny class with label \href{http://www.lmfdb.org/Variety/Abelian/Fq/2/13/a_al}{2.13.a\_al} determined by $t^4 -11 t^2 + 13^2$ and defined over $\F_{13}$    there are abelian varieties admitting a polarisation of degree $4$ but none of them has maximal endomorphism ring.
\end{example}

\begin{example}\label{ex:table}
    Using the aforementioned algorithm, we count the number of polarisations of degree $4$, for every isogeny class in $\PirrA\sqcup\PirrB$ for a fixed range of $q$.
    For the full output, see \url{https://raw.githubusercontent.com/stmar89/Genus3Data/main/table_output.txt}.
    We remark that in the maximal endomorphism ring case, the ratio of isomorphism classes of abelian surfaces admitting a polarisation of degree $4$ is always $0$ or a power of $1/2$. 
\end{example}

\section{Genus \texorpdfstring{$3$}{3} curves lying on abelian surfaces}\label{ses:curveside}
In this final section, we switch our attention from abelian surfaces to curves, and gather information on absolutely irreducible smooth genus $3$ curves lying on simple abelian surfaces. The following lemma characterises such curves as bielliptic.

\begin{lemma}\label{lem:jac}
    Let $q$ be the power of an odd prime. Let $C$ be an absolutely irreducible smooth genus $3$ curve lying on a simple abelian surface $A$, both defined over $\Fq$. Then there exists an elliptic curve $E$ defined over $\F_{q^n}$, with $n\leq 4$, such that $C$ is a double $\F_{q^n}$-cover of $E$.
\end{lemma}
\begin{proof}
    Here we mainly follow \cite[Sec.~1.2]{Barth}, with some minor modifications since we are working over finite fields instead of over $\mathbb{C}$.
    By \cref{thm:pol_deg_4_iff_pa3}, the curve $C$ induces a polarisation $\lambda_{\mathcal{L}}:A\to A^\vee$ of degree $4$, where $\mathcal{L}=\mathcal{L}(C)$ is the ample line bundle defined by $C$. Let $\mathcal{B}$ be the base locus of $\mathcal{L}$, that is, $\mathcal{B} = \{P\in A \,\mid\, s(P)=0 \text{ for any section } s\in H^0(A,\mathcal{L})\}$ \cite[II, Lem.~7.8]{Hart}. Fix some $P_0\in \mathcal{B}$. Since $\ker \lambda_{\mathcal{L}}= \{P\in A \,\mid\, t_P^* \mathcal{L}\simeq \mathcal{L}\}$ \cite[Ch.~I, Sec.~8]{MilneAV} we have, for any and $P\in \ker \lambda_{\mathcal{L}}$, $s(P_0+P)=t_P^* s(P_0)=s(P_0)=0$. Thus $P_0+ \ker\lambda_{\mathcal{L}}\subseteq  \mathcal{B}$. Since we are in characteristic different from $2$, $ \ker \lambda_{\mathcal{L}}$ has order equal to $4$, the degree of $\lambda_{\mathcal{L}}$. On the other hand, $\mathcal{B}$ contains at most $C.C=4$ points. Hence, $ P_0+\ker\lambda_{\mathcal{L}}$ and  $\mathcal{B}$ coincide as sets of geometric points. Consider now the involution $\iota: A\to A$ given by $P\mapsto -P+2P_0$. Since the points in $\ker \lambda_{\mathcal{L}}$ are $2$-torsion points, the points in $\mathcal{B}$ are fixed points of $\iota$. 
    Furthermore, we have $\iota^*\mathcal{L}\simeq \mathcal{L}$, and in particular $\iota C = C$ \cite[Prop.~1.6]{Barth}.  Then one can lift $\iota$ to $C$ as an involution, and since $C$ is smooth, $\iota$ has exactly four fixed points on $C$, namely the points in $\mathcal{B}$  \cite[1.7]{Barth}. Then, by the genus formula, the quotient $E= C/\iota$ is an elliptic curve, that is, $C$ is a double cover of $E$. Note that $\iota$ is not necessarily an $\Fq$-rational map: this will depend on whether  $\mathcal{B}$ contains a rational point or not, that is, whether one can take $P_0$ to be rational.  However, the fact that $\mathcal{B}$ is rational as a set implies that it is a (necessarily disjoint) union of orbits under the action of the Frobenius endomorphism. A point is of degree $n$ if the size of its orbit is $n$. As the size of $\mathcal{B}$ is $4$, the degree of its points is at most $4$. Hence the point $P_0$ defining the involution $\iota$ will have degree at most $4$. So $\iota$ and  $E$ are defined over an extension of $\Fq$ of degree at most $4$.  
\end{proof}

For the rest of the section, we fix an abelian surface $A$ defined over $\Fq$ with Weil polynomial in $\PirrA\sqcup\PirrB\sqcup \{ (t^2-2)^2 , (t^2-3)^2 \}$, that is, which is simple and not isogenous to a Jacobian, or equivalently, not containing absolutely irreducible curves of geometric genus $0$, $1$ or $2$ by Theorem~\ref{thm:no_g1g2}. We suppose that there is an absolutely irreducible smooth genus $3$ curve $C$ defined over $\Fq$ that lies on $A$, and we want to study that curve.

We recall that absolutely irreducible smooth genus $3$ curves are either hyperelliptic curves or plane quartics.  The next lemma shows that, under some conditions, hyperelliptic genus $3$ curves cannot lie on our abelian surfaces.

\begin{lemma}\label{lem:nohyper}
    Let $q$ be the power of an odd prime. Let $C$ be an absolutely irreducible smooth genus $3$ curve lying on an abelian surface $A$ defined over $\Fq$ which is simple and not isogenous to a Jacobian. Suppose further that the elliptic curve doubly covered by $C$ from Lemma \ref{lem:jac} is defined over $\Fq$. Then $C$ is not hyperelliptic.
\end{lemma}
\begin{proof}
    Suppose by contradiction that $C$ is hyperelliptic. We mainly follow the reasoning in the introduction of \cite{RR18}. By Lemma \ref{lem:jac} and our hypothesis, $C$ is also bielliptic over $\Fq$. Then, it is of the form $y^2=x^8+ax^6+bx^4+cx^2+1=f(x^2)$, with $a,b,c\in \Fq$. Here $f$ is a polynomial of degree $4$. In this case, the double cover is given by the involution $(x,y)\mapsto (-x,y)$ and the elliptic curve is given by $y^2=f(x)$. Furthermore, quotienting $C$ by the hyperelliptic involution $(x,y)\mapsto (-x,-y)$ gives the genus $2$ curve $F\colon y^2=x \cdot f(x)$. Finally, we have $\Jac(C)\sim E \times \Jac(F)$. On the other hand, as $C$ lies on $A$, we know by \cite[Prop.~2]{Haloui17}\footref{footnote} that $A$ is an isogeny factor of $\Jac(C)$, that is, there exists an elliptic curve $E'$ such that $\Jac(C)\sim A \times E'$. Since $A$ is simple we necessarily have $E\sim E'$, and hence $\Jac(F)\sim A$. The Jacobian of $F$ can be simple or split in the product of two elliptic curves. In both cases, this leads to a contradiction.
\end{proof}

Hence, we are left with the case in which $C$ is a bielliptic plane quartic. In this case, using results from \cite{RR18}, we  get the following result. 

\begin{proposition} 
    Let $q$ be the power of an odd prime. Let $C$ be an absolutely irreducible smooth genus $3$ curve lying on an abelian surface defined over $\Fq$  respecting one of the equivalent conditions of Theorem~\ref{thm:no_g1g2}. Suppose further that the elliptic curve doubly covered by $C$ from Lemma \ref{lem:jac} is defined over $\Fq$. 
    Then, $C$ is a plane quartic of the form $y^4-h(x,z)y^2+r(x,z)=0$, where $h$ and $r$ are homogenous polynomials of degree $2$ and $4$, respectively, and one of the followings holds:
    \begin{enumerate}[(i)]
        \item the polynomial $r(x,z)$ cannot be decomposed (over $\Fq$) as the product of two polynomials $f,g$ of degree $2$. In particular, the polynomial $r$ is either irreducible or it has only one root in $\Fq$;
        \item we have $r(x,z)=f(x,z)\cdot g(x,z)$ with $\deg f = \deg g =2$, and the polynomial $h(x,z)$ is a linear combination of $f$ and $g$.
    \end{enumerate}
\end{proposition}
\begin{proof}
    By Lemma \ref{lem:nohyper} we know that $C$ cannot be hyperelliptic. Applying Lemma \ref{lem:jac} and using the hypothesis of the proposition, $C$ is the double cover of an elliptic curve $E$ defined over $\Fq$. Then, we can assume that the involution giving the double cover is $(x:y:z)\mapsto (x:-y:z)$, and therefore that $C$ can be written in the form $y^4-h(x,z)y^2+r(x,z)=0$, where $h$ is a homogenous degree $2$ polynomial, and $r$ is homogenous of degree $4$. Let us suppose that the polynomial $r$ has two quadratic factor $f$ and $g$ and that the polynomial $h(x,z)$ is not a linear combination of $f$ and $g$. Then, by  \cite[Thm.~1.1]{RR18} we can explicitly construct an absolutely irreducible smooth genus $2$ curve $F$ such that $\Jac(C) \sim \Jac(F)\times E$.  Then, the curve $C$ cannot lie on $A$, since otherwise $A$ would be isogenous to $\Jac(F)$ by Lemma \ref{lem:jac}. Finally, one of the two conditions in the proposition must hold.
\end{proof}

Finally, we highlight that if an absolutely irreducible genus $3$ curve can be embedded in an abelian surface with Weil polynomial in $\PirrA\sqcup\PirrB\sqcup  \{ (t^2-2)^2 , (t^2-3)^2 \}$, then its number of rational points is far from the Serre-Weil bound.  This illustrates how embedding entails arithmetic constraints on curves. 
To start with, we recall and extend a result from \cite{Haloui17} on the number of rational points on curves over abelian surfaces.
\begin{proposition}\label{prop:rationalpoints}
    Let $A$ be an abelian surface defined over $\Fq$ with trace $-a$, that is, with Weil polynomial $t^4+at^3+bt^2+qat+q^2$. Let $C$ be an absolutely irreducible curve defined over $\Fq$, of arithmetic genus $p_a(C)=p_a$, lying on $A$. Then
    \[ \vert\#C(\Fq)  - (q + 1+a) \vert \leq |p_a-2|\lfloor 2\sqrt{q}\rfloor.\]
\end{proposition}
\begin{proof}
    The upper bound on $\#C(\Fq)$ is \cite[Thm.~4]{Haloui17}. 
    The proof can be adapted to obtain also the lower bound on $\#C(\Fq)$. In what follows, we borrow the notation from the proof of \cite[Thm.~4]{Haloui17}.

    If $p_a=1$, then we know that the curve $C$ is elliptic, of trace say $-e$. Write $a=e+x_2$ for some integer $-\lfloor 2\sqrt{q}\rfloor\leq x_2 \leq \lfloor 2\sqrt{q}\rfloor$. Hence,
    \[ \#C(\Fq) = q+1 +e = q+1+a-x_2 \geq q+1+a - \lfloor 2\sqrt{q}\rfloor. \]

    Suppose now $p_a\geq 2$. 
    By \cite[Prop.~2.3]{AP96}, we know that $\#C(\Fq)-\# \tilde C(\Fq) \geq g-p_a$, where $\tilde C$ is the normalisation of $C$ and $g$ is its genus. Following the reasoning in the proof of  \cite[Thm.~4, Eq.~(5)]{Haloui17}, there exist real numbers $x_3,\dots,x_g$ such that
    $$\# \tilde C(\Fq)=q+1+a+\sum_{i=3}^g x_i\,\text{ and }\,\sum_{i=3}^g x_i \geq -(g-2)\lfloor 2\sqrt{q}\rfloor.$$ Using that  $g-p_a\leq 0$, we get the following series of inequalities:
    \begin{align*}
    \#C(\Fq)& \geq q+1+a+\sum_{i=3}^g x_i+g-p_a\\
    & \geq q+1+a - (g-2)\lfloor 2\sqrt{q}\rfloor +g -p_a\\
    & \geq q+1+a - (g-2)\lfloor 2\sqrt{q}\rfloor +(g - p_a)\lfloor 2\sqrt{q}\rfloor\\
    &= q+1+a - (p_a-2)\lfloor 2\sqrt{q}\rfloor,
    \end{align*}
    and the statement follows.
\end{proof}

Note that in \cite[Thm.~4]{Haloui17} the author needs the hypothesis $a\geq -q$, which is however only used  to prove the result for curves which are $\Fq$-irreducible but not absolutely irreducible. For such curves, the number of rational points cannot exceed $p_a-1$, but we cannot prove a lower bound different from the trivial one.

From Proposition \ref{prop:rationalpoints}, we deduce the following facts. First, recall that a Weil restriction has always zero trace. Thus, if $C$ is an absolutely irreducible curve of arithmetic genus $3$ lying on a Weil restriction, then
\begin{equation}\label{eq:bound_weilrestr}
    q + 1 - \lfloor 2\sqrt{q}\rfloor\leq\#C(\Fq)\leq  q + 1 + \lfloor 2\sqrt{q}\rfloor.
\end{equation}
Indeed, in this case, $C$ has the same number of rational points of the elliptic curve of which it is the double cover.

Secondly, we know that the trace of an abelian surface which does not admit a principal polarisation respects $a^2=q-b$. Then, if $C$ is an absolutely irreducible curve of arithmetic genus $3$ on such a surface, we have
\[q + 1 -\sqrt{q-b}- \lfloor 2\sqrt{q}\rfloor\leq\#C(\Fq)\leq  q + 1 +\sqrt{q-b} +\lfloor 2\sqrt{q}\rfloor.\]
A sloppy estimation, using that $|b|\leq 2q$ and hence $\sqrt{q-b}\leq \sqrt{3q}\leq 2 \sqrt{q}$, gives us
\begin{equation}\label{eq:bound_nopp}
    q + 1 - 2\lfloor 2\sqrt{q}\rfloor\leq\#C(\Fq)\leq  q + 1 + 2\lfloor 2\sqrt{q}\rfloor.
\end{equation}

In particular, we see that an absolutely irreducible curve of genus $3$ lying on a simple abelian surface not isogenous to the Jacobian of an absolutely irreducible smooth genus $2$ curve is far from reaching the Serre--Weil bound which states $|\#C(\Fq)-  (q + 1)|\leq 3\lfloor 2\sqrt{q}\rfloor$.

\section*{Acknowledgments}
The first author would like to thank Christophe Ritzenthaler for preliminary discussions about the central question of this paper back in 2019, and for putting her in contact with the third author.  The first and third authors thank Qing Liu for some useful discussion about curves on abelian surfaces during the CAVARET conference. The third author thanks Elisa Lorenzo Garc\'ia for valuable conversations.
The authors are grateful to Jonas Bergstr\"om and Gaetan Bisson for comments on a preliminary version of the paper. Finally, the authors appreciate the valuable insights provided by Jean Kieffer and Damien Robert during the CAIPI symposium in Nancy, which led to the final version of Section \ref{ses:curveside}.

The first author is supported by the grants ANR-21-CE39-0009-BARRACUDA, ANR-22-CPJ2-0047-01, and PEPS “Jeunes chercheurs et jeunes chercheuses" 2025 from Insmi. 
The third author is supported by Nederlandse Organisatie voor Wetenschappelijk Onderzoek, grant number VI.Veni.202.107, by
Agence Nationale de la Recherche under the MELODIA project, grant number ANR-20-CE40-0013, and by
Marie-Sk{\l}odowska-Curie Actions - Postdoctoral Fellowships 2023 (project 101149209 - AbVarFq).
   
\bibliographystyle{siam}
\bibliography{references_ASG3}
\end{document}